\def\benm{\begin{enumerate}}
\def\eenm{\end{enumerate}}

\def\bal{\begin{align}}
\def\eal{\end{align}}

\def\ii{\mathrm{i}}

\documentclass[11pt]{amsart}
\usepackage{amsmath, amsthm, amscd, amsfonts,amssymb, graphicx}
\usepackage{color}
\usepackage{cite}
\usepackage{hyperref}

\newtheorem{theorem}{Theorem}[section]

\theoremstyle{definition}

\newtheorem{example}[theorem]{Example}

\newtheorem{proposition}[theorem]{Proposition}
\newtheorem{corollary}[theorem]{Corollary}

\theoremstyle{remark}
\newtheorem{remark}[theorem]{Remark}

\numberwithin{equation}{section}

\newcommand{\bfomega}{\mbox{\boldmath $\omega$ \unboldmath} \hskip -0.05 true in}
\newcommand{\bom}{\bfomega}

\newcommand{\bfmu}{\mbox{\boldmath $\mu$ \unboldmath} \hskip -0.05 true in}

\newcommand{\dd}{\mathrm{d}}
\newcommand{\Sp}{\mathrm{Sp}}
\newcommand{\bfS}{\mathbf{S}}

\begin{document}

\title[Covariant Function Algebras of Invariant Characters]{Covariant Function Algebras of Invariant Characters of Normal Subgroups}

\author[A. Ghaani Farashahi]{Arash Ghaani Farashahi}
\address{Department of Mechanical Engineering, National University of Singapore,
9 Engineering Drive 1, Singapore 117575.}
\email{arash.ghaanifarashahi@nus.edu.sg}
\email{ghaanifarashahi@outlook.com}
\address{Department of Pure Mathematics, School of Mathematics, University of Leeds, Leeds LS2 9JT, United Kingdom}
\email{a.ghaanifarashahi@leeds.ac.uk}

\curraddr{}




\subjclass[2020]{Primary 43A15, 43A20, 43A85.}

\date{\bf June 18, 2023}


\keywords{covariance property, covariant function, normal subgroup, covariant convolution, covariant involution, character, invariant character, twisted convolution, twisted involution.}

\begin{abstract}
This paper presents abstract harmonic analysis foundations for structure of  covariant function algebras of invariant characters of normal subgroups. 
Suppose that $G$ is a locally compact group and $N$ is a closed normal subgroup of $G$. Let $\xi:N\to\mathbb{T}$ be a continuous $G$-invariant character, $1\le p<\infty$, and $L_\xi^p(G,N)$ be the $L^p$-space of all covariant functions of $\xi$ on $G$. We study structure of covariant convolution in $L^p_\xi(G,N)$. It is proved that $L^1_\xi(G,N)$ is a Banach $*$-algebra and $L^p_\xi(G,N)$ is a Banach $L^1_\xi(G,N)$-module. We then investigate the theory of covariant convolutions for the case of characters of canonical normal subgroups in semi-direct product groups. The paper is concluded by realization of the theory in the case of different examples. 
\end{abstract}

\maketitle

\section{\bf{Introduction}}
The Banach convolution function algebras play building block roles in abstract harmonic analysis, functional analysis, and operator theory, see \cite{AGHF.IJM, kan.lau}. Over the last decades, different aspects of Banach convolution function algebras have achieved significant popularity in postmodern areas of harmonic analysis such as coorbit theory and constructive approximation, see \cite{Fei1, Fei2, kisil.cov1}. The covariant functions of characters of closed subgroups give a unified formulation for interesting mathematical objects in variant topics including calculus of pseudo-differential operators, number theory (automorphic forms), induced representations, homogeneous spaces, complex (hypercomplex) analysis, and coherent states, see \cite{RB, FollP, AGHF.CJM, kan.taylor, kisil.cov2, kisil1, MackII, MackI, pere}.

The classical methods of abstract harmonic analysis on groups cannot be applied for covariant functions as a unified theory. In the case of covariant functions of characters of compact subgroups, harmonic analysis on covariant functions studied in \cite{AGHF.ch.com}. Some fundamental operator theoretic aspects for Banach function spaces of covariant functions of characters of closed normal subgroups studied in \cite{AGHF.ch.normal.main}. The convolution module action of group algebras on Banach functions spaces of covariant functions of closed normal subgroups investigated in \cite{AGHF.ch.normal.mod}. This paper gives a unified and systematic approach to the structure of covariant convolution function algebras of invariant characters of closed normal subgroups.  The introduced technique has several advantages.  To begin with,  it is independent of assuming existence and employment of Borel cross-sections.  In addition, it gives a universal characterization for fundamental structures of abstract harmonic analysis on locally compact groups in different directions.  In the framework of twisted convolutions using multipliers,  it gives a universal characterization of certain twisted convolutions \cite{MackActa1958, LepMAn1965,LepInv1-1967, LepInv2-1967,LepInv1968,BuSmTrans1970}. In particular,  it extends the classical theory for harmonic analysis of convolution function algebras on quotient groups by assuming the character to be the trivial character of the normal subgroup \cite{der00, der1983}. 

This article contains 5 sections and organized as follows. Section  2  is  devoted  to  fix  notations  and  provides  a  summary
of classical harmonic analysis on locally compact groups, quotient groups of locally compact groups, and basics for covariant functions of characters. 
Let $G$ be a locally compact group and $N$ be a closed normal subgroup of $G$. Suppose $\xi:N\to\mathbb{T}$ is a $G$-invariant character of $N$. 
In section 3 we introduce the abstract notions of convolution and involution on covariant function spaces of the character $\xi$. It is shown that if $N$ is compact then covariant convolution (resp. involution) coincides with the convolution (resp. involution) on $G$. It is proved that $L^1_\xi(G,N)$ is a Banach $*$-algebra and $L^p_\xi(G,N)$ is a Banach $L^1_\xi(G,N)$-module.
Section 4 investigates covariant convolutions/involutions for semi-direct products. The paper is concluded by illustration of the theory for the case of some examples including abstract Weyl-Heisenberg groups and subgroups of Jacobi groups.  

\section{\bf{Preliminaries and Notations}}

Throughout this section we present basic preliminaries and notations.

Let $\mathcal{A}$ be a Banach $*$-algebra. A right approximate identity in 
$\mathcal{A}$ is a net $\mathcal{E}=(f_\alpha)_{\alpha\in\Lambda}$ such that for every element $f$ of $\mathcal{A}$, we have $f=\lim_\alpha f\ast f_\alpha$. Similarly, $\mathcal{E}$ is called a left approximate identity for the Banach algebra $\mathcal{A}$, if $f=\lim_\alpha f_\alpha\ast f$ for all $f\in\mathcal{A}$. A two-sided approximate identity for $\mathcal{A}$ is a net $\mathcal{E}=(f_\alpha)_{\alpha\in\Lambda}$ which is both a right approximate identity and a left approximate identity for $\mathcal{A}$. If an approximate identity is a sequence, it is called a sequential approximate identity.

Suppose that $X$ is a locally compact Hausdorff space. Then $\mathcal{C}(X)$ is the space of all continuous complex valued functions on $X$ and $\mathcal{C}_c(X)$ is the space of all continuous complex valued functions on $X$ with compact support. Let $\lambda$ be a positive Radon measure on $X$, and $1\le p<\infty$. We then denote the Banach space consists of equivalence classes of $\lambda$-measurable complex valued functions $f:X\to\mathbb{C}$ such that
$$\|f\|_{L^p(X,\lambda)}^p:=\int_X|f(x)|^p\dd\lambda(x)<\infty,$$
by $L^p(X,\lambda)$ which contains $\mathcal{C}_c(X)$ as a $\|\cdot\|_{L^p(X,\lambda)}$-dense subspace.

\subsection{Abstract Harmonic Analysis on Groups}Let $G$ be a locally compact group with the modular function $\Delta_G$ and a left Haar measure $\lambda_{G}$.
For $1\le p<\infty$, the notation $L^p(G)$ used for the Banach function space $L^p(G,\lambda_G)$. The convolution of functions $f,g\in\mathcal{C}_c(G)$ is defined by 
\begin{equation}\label{f.ast.g}
f\ast_G g(x):=\int_Gf(y)g(y^{-1}x)\dd\lambda_G(y)\quad (x\in G).
\end{equation}
Then $f\ast_G g\in\mathcal{C}_c(G)$ and the bilinear product $\ast_G:\mathcal{C}_c(G)\times\mathcal{C}_c(G)\to\mathcal{C}_c(G)$ has a unique extension to the bilinear product $\ast_G:L^1(G)\times L^p(G)\to L^p(G)$ such that 
$\|f\ast_G g\|_{L^p(G)}\le \|f\|_{L^1(G)}\|g\|_{L^p(G)}$, if $f\in L^1(G)$ and $g\in L^p(G)$. The Banach function space $L^1(G)$ is a Banach $*$-algebra with respect to the bilinear product $\ast_G:L^1(G)\times L^1(G)\to L^1(G)$ and the involution given by $f\mapsto f^{\ast^G}$ where $f^{*^G}(x):=\Delta_G(x^{-1})\overline{f(x^{-1})}$ for $x\in G$. The Banach $*$-algebra $L^1(G)$ admits a two-sided approximate identity. If $G$ is second-countable then $L^1(G)$ admits a sequential approximate identity. For $p>1$, the Banach space $L^p(G)$ is a Banach left $L^1(G)$-module equipped with  the left module action $\ast_G:L^1(G)\times L^p(G)\to L^p(G)$, see \cite{FollH, HR1, 50} and the list of references therein. 

Let $N$ be a closed normal subgroup of $G$ with the left Haar measure $\lambda_N$. The coset space $G/N$ is a locally compact group called as quotient (factor) group of $N$ in $G$. Then $\mathcal{C}_c(G/N)$ consists of all functions $T_N(f)$, where 
$f\in\mathcal{C}_c(G)$ and
\begin{equation}\label{5.1}
T_N(f)(xN)=\int_Nf(xs)\dd\lambda_{N}(s)\quad (xN\in G/N).
\end{equation}
The quotient group $G/N$ has a left Haar measure $\lambda_{G/N}$, normalized with respect to the following Weil's formula  
\begin{equation}\label{6}
\int_{G/N}T_N(f)(xN)\dd\lambda_{G/N}(xN)=\int_Gf(x)\dd\lambda_G(x),
\end{equation}
for $f\in L^1(G)$, see \cite{der1983, FollH}.
Then the convolution of $\Psi\in L^1(G/N)$ and $\Phi\in L^p(G/N)$ has the form 
\begin{equation*}
\Psi\ast_{G/N}\Phi(xN)=\int_{G/N}\Psi(yN)\Phi(y^{-1}xN)\dd\lambda_{G/N}(yN),
\end{equation*}
and the involution is given by 
$\Psi^{\ast^{G/N}}(xN):=\Delta_{G/N}(x^{-1}N)\overline{\Psi(x^{-1}N)}$
for $xN\in G/N$. In this case, the linear map $T_N:L^1(G)\to L^1(G/N)$ is a contraction (norm-decreasing) $*$-homomorphism of Banach $*$-algebras, see Theorem 3.5.4 of \cite{50}.   

\subsection{Covariant Functions of Characters}
Let $G$ be a locally compact group and $N$ be a closed subgroup of $G$. 
A character $\xi$ of $N$, is a continuous group homomorphism $\xi:N\to\mathbb{T}$, where $\mathbb{T}:=\{z\in\mathbb{C}:|z|=1\}$ is the circle group. In terms of group representation theory, each character of $N$ is a 1-dimensional irreducible continuous unitary representation of $N$. We then denote the set of all characters of $N$ by $\chi(N)$. 
A function $\psi:G\to\mathbb{C}$ satisfies covariant property associated to the character $\xi\in\chi(N)$, if $\psi(xs)=\xi(s)\psi(x)$, for every $x\in G$ and $s\in N$. In this case, $\psi$ is called a covariant function of $\xi$. The covariant functions appear in abstract harmonic analysis in the construction of induced representations, see \cite{FollH, kan.taylor}. We here employ some of the classical tools in this direction. Suppose that $\lambda_N$ is a left Haar measure on $N$. For $\xi\in\chi(N)$ and $f\in\mathcal{C}_c(G)$, let $T_\xi(f):G\to\mathbb{C}$ be the function defined by  
\[
T_\xi(f)(x):=\int_Nf(xs)\overline{\xi(s)}\dd\lambda_N(s),
\]
for every $x\in G$. Assume that $\mathcal{C}_\xi(G,N)$ is the linear subspace of $\mathcal{C}(G)$ given by 
\[
\mathcal{C}_\xi(G,N):=\{\psi\in\mathcal{C}_c(G|N):\psi(xk)=\xi(k)\psi(x),\ {\rm for\ all}\ x\in G,\ k\in N\},
\]
where 
\[
\mathcal{C}_c(G|N):=\{\psi\in\mathcal{C}(G):\mathrm{q}({\rm supp}(\psi))\ {\rm is\ compact\ in}\ G/N\},
\]
and $\mathrm{q}:G\to G/N$ is the canonical map given by $\mathrm{q}(x):=xN$ for every $x\in G$. The linear operator $T_\xi$ maps $\mathcal{C}_c(G)$ onto $\mathcal{C}_\xi(G,N)$, see Proposition 6.1 of \cite{FollH}.
If $N$ is a compact subgroup of $G$ then $\mathcal{C}_\xi(G,N)\subseteq\mathcal{C}_c(G)$, see Proposition 3.1 of \cite{AGHF.ch.com}. 

\subsection{\bf Covariant Functions of Characters of Normal Subgroups}
We then review some aspects of harmonic analysis of covariant functions of characters of normal subgroups on locally compact groups. Throughout, let $G$ be a locally compact group and $N$ be a closed normal subgroup of $G$.
Then there exists a unique homomorphism $\sigma_N:G\to (0,\infty)$ such that 
\begin{equation}\label{tau.main}
\int_Nv(s)\dd\lambda_N(x^{-1}sx)=\sigma_N(x)\int_Nv(s)\dd\lambda_N(s),
\end{equation}
and 
\[
\int_Nv(s)\dd\mu_N(x^{-1}sx)=\sigma_N(x)\int_Nv(s)\dd\mu_N(s),
\]
for every left Haar measure $\lambda_N$ of $N$, right Haar measure $\mu_N$ of $N$, $v\in\mathcal{C}_c(N)$, and $x\in G$. It is worthwhile to mention that existence of the homomorphism $\sigma_N:G\to (0,\infty)$ guaranteed by the uniqueness of left Haar measures on the locally compact group $N$, Theorem 2.20 of \cite{FollH}. We then have $\sigma_N(t)=\Delta_N(t)$ for all $t\in N$ and hence $\sigma_G(x)=\Delta_G(x)$ for $x\in G$. The modular function of $G/N$ satisfies $\Delta_G(x)=\sigma_N(x)\Delta_{G/N}(xN)$ for $x\in G$, see Proposition 11 of \cite[Chap. VII, \S2]{Bour}. In particular, if $N$ is central we then have $\sigma_N=1$.

Suppose that $\xi\in\chi(N)$ is given. Let $\psi:G\to\mathbb{C}$ be a covariant function of $\xi$. Then  
\begin{equation}\label{cov.normal}
\psi(sx)=\xi(x^{-1}sx)\psi(x),
\end{equation} 
for $x\in G$ and $s\in N$. Also, according to Proposition 3.3 of \cite{AGHF.ch.normal.main}, for $x\in G$, and $f\in\mathcal{C}_c(G)$, we have  
\begin{equation}\label{J.xi.normal.1}
T_\xi(f)(x)=\sigma_N(x)\int_Nf(sx)\overline{\xi(x^{-1}sx)}\dd\lambda_N(s).
\end{equation}

We then review structure of classical $L^p$-spaces of covariant functions of characters of normal subgroups. Let $\lambda_{G}$ be a left Haar measure on $G$ and $\lambda_N$ be a left Haar measure on $N$.
Suppose that $\lambda_{G/N}$ is the left Haar measure on the quotient group $G/N$ normalized with respect to Weil's formula (\ref{6}). For every $1\le p<\infty$ and $\psi\in \mathcal{C}_\xi(G,N)$, the covariant norm is defined by 
\[
\|\psi\|_{(p)}^p:=\int_{G/N}|\psi(y)|^p\dd\lambda_{G/N}(yN).
\]
We then assume that $L^p_\xi(G,N)$ is the Banach completion of the normed linear space $\mathcal{C}_\xi(G,N)$ with respect to $\|\cdot\|_{(p)}$. The  completion norm is also denoted by $\|\cdot\|_{(p)}$. Invoking the structure of $L^p_\xi(G,N)$, as the Banach completion of a normed space, one can identify the abstract space $L^p_\xi(G,N)$ with a linear space of complex-valued functions on $G$, where two functions are identified when they are equal locally almost everywhere, see Remark 3.6 of \cite{AGHF.ch.normal.mod}. 

It is shown that the linear operator $T_\xi:(\mathcal{C}_c(G),\|\cdot\|_{L^1(G)})\to(\mathcal{C}_\xi(G,N),\|\cdot\|_{(1)})$ 
is a contraction, see Theorem 4.4 of \cite{AGHF.ch.normal.main}. 
Therefore, $T_\xi:(\mathcal{C}_c(G),\|\cdot\|_{L^1(G)})\to(\mathcal{C}_\xi(G,N),\|\cdot\|_{(1)})$ has a unique extension to a contraction from $L^1(G)$ onto $L^1_\xi(G,N)$. The extended map $T_\xi:L^1(G)\to L^1_\xi(G,N)$ is given by $f\mapsto T_\xi(f)$, where  (Theorem 5.4 of \cite{AGHF.ch.normal.main})
\[
T_\xi(f)(x)=\int_Nf(xs)\overline{\xi(s)}\dd\lambda_N(s)\ \ \ {\rm for\ \ l.a.e.}\ x\in G.
\]
If $\mathcal{N}_\xi^1(G,N)$ is the kernel of the linear operator $T_\xi$ in $L^1(G)$ then the Banach space $L_\xi^1(G,N)$ is canonically isometric isomorphic to the quotient Banach space $L^1(G)/\mathcal{N}_\xi^1(G,N)$, see Theorem 5.9 of \cite{AGHF.ch.normal.main}.

If $N$ is a compact normal subgroup of $G$, each Haar measure of $N$ is finite and hence 
\begin{equation}\label{p.comp}
\|\psi\|_{(p)}=\lambda_N(N)^{-1/p}\|\psi\|_{L^p(G)},  
\end{equation}
for every $\psi\in \mathcal{C}_\xi(G,N)$, see Proposition 3.2 of \cite{AGHF.ch.normal.mod}. Therefore, $L^p_\xi(G,N)\subseteq L^p(G)$ when $N$ is compact. 
It is proved that if $N$ is a compact subgroup of $G$ 
with the probability Haar measure $\lambda_N$ on $N$ and $p>1$, then the linear operator $T_\xi:(\mathcal{C}_c(G),\|.\|_{L^p(G)})\to(\mathcal{C}_\xi(G,N),\|.\|_{L^p(G)})$ is a contraction, see Theorem 4.1 of \cite{AGHF.ch.com}.
We then denote by $T_\xi:L^p(G)\to L^p_\xi(G,N)$, the unique extension of $T_\xi:(\mathcal{C}_c(G),\|.\|_{L^p(G)})\to(\mathcal{C}_\xi(G,N),\|.\|_{L^p(G)})$ to a contraction from $L^p(G)$ onto $L^p_\xi(G,N)$. In this case, the Banach space $L_\xi^p(G,N)$ is isometrically isomorphic to the quotient Banach $L^p(G)/\mathcal{N}_\xi^p(G,N)$, where $\mathcal{N}^p_\xi(G,N)$ is the kernel of the linear operator $T_\xi:L^p(G)\to L_\xi^p(G,N)$ in $L^p(G)$, see Theorem 4.9 of \cite{AGHF.ch.com}.

\section{\bf Covariant Function Algebras of Invariant Characters}
In this section, we study structure of covariant function algebras of invariant characters of normal subgroups. We introduce the abstract notions of convolutions and involutions for covariant functions of invariant characters of normal subgroups. Suppose that $G$ is a locally compact group and $N$ is a closed normal subgroup of $G$. Let $\lambda_{G}$ be a left Haar measure on $G$ and $\lambda_N$ be a left Haar measure on $N$.
Assume that $\lambda_{G/N}$ is the left Haar measure on the 
quotient group $G/N$ normalized with respect to Weil's formula (\ref{6}).
  
\subsection{Covariant Convolutions}
Throughout, we present the abstract notion of covariant convolutions,  convolution of covariant functions, for invariant characters of normal subgroups. 

A character $\xi\in\chi(N)$ is called {\it $G$-invariant} if $\xi_x=\xi$ for every $x\in G$, where the character $\xi_x:N\to\mathbb{T}$ is given by 
$\xi_x(s):=\xi(x^{-1}sx)$, for every $s\in N$. We then denote the set of all $G$-invariant characters of $N$ by $\Gamma(G,N)$, that is 
\[
\Gamma(G,N):=\{\xi\in\chi(N):\xi_x=\xi,\ {\rm for\ all}\ x\in G\}.
\]
If $N$ is a closed central 
subgroup of $G$ then $\Gamma(G,N)=\chi(N)$.

\begin{remark}\label{left.cov.inv.xi}
Let $N$ be a closed normal subgroup of $G$, and $\xi\in\Gamma(G,N)$. Suppose that $\psi:G\to\mathbb{C}$ is a covariant function of $\xi$. Then  $\psi(sx)=\xi(s)\psi(x)$ for $x\in G$ and $s\in N$. Indeed, 
using (\ref{cov.normal}) and since $\xi$ is $G$-invariant, we have 
\[
\psi(sx)=\xi(x^{-1}sx)\psi(s)=\xi(s)\psi(x).
\]
\end{remark}

We then have the following observations concerning invariant characters. 

\begin{proposition}\label{J.xi.inv.betaN}
{\it Let $G$ be a locally compact group and $N$ be a closed normal subgroup of $G$. Suppose $\xi\in\Gamma(G,N)$, and $f\in\mathcal{C}_c(G)$ are arbitrary. Then
\begin{enumerate}
\item For every $x\in G$, we have  
\[
T_\xi(f)(x)=\sigma_N(x)\int_Nf(sx)\overline{\xi(s)}\dd\lambda_N(s).
\]
\item For every $x\in G$ and $k\in N$, we have $T_\xi(L_kf)(x)=\overline{\xi(k)}T_\xi(f)(x)$.
\end{enumerate}
}\end{proposition}
\begin{proof}
(1) Let $x\in G$ be given. Using (\ref{J.xi.normal.1}) and since $\xi$ is $G$-invariant, we have 
\[
T_\xi(f)(x)=\sigma_N(x)\int_Nf(sx)\overline{\xi(x^{-1}sx)}\dd\lambda_N(s)=\sigma_N(x)\int_Nf(sx)\overline{\xi(s)}\dd\lambda_N(s).
\]
(2) Invoking Proposition 3.4 of \cite{AGHF.ch.normal.main}, we get  
\[
T_\xi(L_kf)(x)=\overline{\xi(x^{-1}kx)}T_\xi(f)(x)=\overline{\xi(k)}T_\xi(f)(x).
\]
\end{proof}

\begin{proposition}\label{main.wd}
{\it Let $G$ be a locally compact group and $N$ be a closed normal subgroup of $G$. Suppose $\xi\in\Gamma(G,N)$, $x\in G$, and $\psi,\phi\in \mathcal{C}_\xi(G,N)$. The mapping $\varphi_x:G\to\mathbb{C}$ given by $\varphi_x(y):=\psi(y)\phi(y^{-1}x)$ reduces to constant on $N$ and hence it defines a well-defined complex-valued function on $G/N$.
}\end{proposition}
\begin{proof}
Let $y\in G$ and $s\in N$. Since $N$ is normal and $\xi$ is $G$-invariant, using Remark \ref{left.cov.inv.xi}, we have   
\begin{align*}
\phi(s^{-1}y^{-1}x)
=\xi(s^{-1})\phi(y^{-1}x).
\end{align*}
Therefore, we get 
\begin{align*}
\varphi_x(ys)
&=\psi(ys)\phi((ys)^{-1}x)
=\psi(ys)\phi(s^{-1}y^{-1}x)
\\&=\psi(ys)\xi(s^{-1})\phi(y^{-1}x)
=\xi(s)\psi(y)\xi(s^{-1})\phi(y^{-1}x)
\\&=\xi(s)\psi(y)\overline{\xi(s)}\phi(y^{-1}x)=|\xi(s)|^2\psi(y)\phi(y^{-1}x)=\varphi_x(y).
\end{align*}
\end{proof}
 
Let $\xi\in\Gamma(G,N)$ and $\psi,\phi\in \mathcal{C}_\xi(G,N)$. We then define the {\it covariant convolution} $\psi\natural\phi:G\to\mathbb{C}$ by
\begin{equation}\label{tc.main.central}
\psi\natural\phi(x):=\int_{G/N}\psi(y)\phi(y^{-1}x)\dd\lambda_{G/N}(yN),
\end{equation}
for every $x\in G$. 

Invoking Proposition \ref{main.wd}, for any $x\in G$, the function 
$\varphi_x:G\to\mathbb{C}$ given by $\varphi_x(y):=\psi(y)\phi(y^{-1}x)$ reduces to constant on $N$. Thus it defines a well-defined complex-valued function on $G/N$. Since $\psi,\phi$ are  continuous on $G$, ${\rm supp}(\psi)N$ and ${\rm supp}(\phi)N$ are compact in $G/N$, we deduce that $\varphi_x$ defines a function in $\mathcal{C}_c(G/N)$ which can be integrated on $G/N$ with respect to $\lambda_{G/N}$. Therefore, the right hand side of  (\ref{tc.main.central}) defines a well-defined function on $G$.

\begin{remark}
Let $\xi=1$ be the trivial character of the normal subgroup $N$. Suppose that  
$\psi,\phi\in \mathcal{C}_1(G,N)$ are identified with $\widetilde{\psi},\widetilde{\phi}\in\mathcal{C}_c(G/N)$. We then have $\psi\natural\phi(x)=\widetilde{\psi}\ast_{G/N}\widetilde{\phi}(xN)$ for $x\in G$. 
\end{remark}

\begin{remark}\label{tws.cnv}
The covariant convolution (\ref{tc.main.central}) gives a universal characterization of twisted convolutions on $G/N$ associated to certain class of multipliers,  if there is a Borel cross-section for $G/N$ in $G$.   Let $\mathrm{s}:G/N\to G$ be a Borel cross-section for $G/N$ in $G$.  Suppose that $\mathrm{r}:G\to N$ is given by $\mathrm{r}(x):=\mathrm{s}(xN)^{-1}x$, for every $x\in G$.  Assume that $\beta:G/N\times G/N\to\mathbb{T}$ is given by 
\[
\beta(xN,yN):=\overline{\xi(\mathrm{s}(xyN)^{-1}\mathrm{s}(xN)\mathrm{s}(yN))},
\]
for every $xN,yN\in G/N$.  Then $\beta$ satisfies 
\[
\xi(\mathrm{r}(xy))=\overline{\beta(xN,yN)}\xi(\mathrm{r}(x))\xi(\mathrm{r}(y)),
\]
for every $x,y\in G$, concluding that $\beta$ is a multiplier on $G/N$. We then have  
\[
\psi\natural\phi(x)=\xi(\mathrm{r}(x))\psi_\mathrm{s}\ast_\beta\phi_\mathrm{s}(xN),
\]
for every $\psi,\phi\in\mathcal{C}_\xi(G,N)$ and $x\in G$, where $\psi_\mathrm{s}:=\psi\circ\mathrm{s}$ (resp.  $\phi_\mathrm{s}:=\phi\circ\mathrm{s}$), and $\ast_\beta$ is the twisted convolution associated to $\beta$, see \cite{MackActa1958, LepMAn1965,LepInv1-1967, LepInv2-1967,LepInv1968,BuSmTrans1970}.
\end{remark}

\begin{proposition}\label{cov.conv.Ncom}
{\it Let $G$ be a locally compact group and $N$ be a compact normal subgroup of $G$. Suppose $\xi\in\Gamma(G,N)$ and $\psi,\phi\in \mathcal{C}_\xi(G,N)$. Then
\[
\psi\ast_{G}\phi=\lambda_N(N)\psi\natural \phi.
\]
}\end{proposition}
\begin{proof}
Let $N$ be a compact normal subgroup of $G$. Then, using Proposition 3.1(1) of \cite{AGHF.ch.com}, we get $\mathcal{C}_\xi(G,N)\subseteq\mathcal{C}_c(G)$. Also, using compactness of $N$, each Haar measure of $N$ is finite. Let $x\in G$ be given. Using Proposition \ref{main.wd} and Weil's formula (\ref{6}), we obtain  
\begin{align*}
\psi\ast_G\phi(x)&=\int_G\psi(y)\phi(y^{-1}x)\dd\lambda_{G}(y)
=\int_{G/N}\int_N\psi(ys)\phi(s^{-1}y^{-1}x)\dd\lambda_N(s)\dd\lambda_{G/N}(yN)
\\&=\int_{G/N}\int_N\psi(y)\phi(y^{-1}x)\dd\lambda_N(s)\dd\lambda_{G/N}(yN)
=\int_{G/N}\left(\int_N\dd\lambda_N(s)\right)\psi(y)\phi(y^{-1}x)\dd\lambda_{G/N}(yN)
\\&=\lambda_N(N)\int_{G/N}\psi(y)\phi(y^{-1}x)\dd\lambda_{G/N}(yN)
=\lambda_N(N)\psi\natural\phi(x).
\end{align*}
\end{proof}

\begin{remark}
Let $N$ be a compact normal subgroup of $G$. Suppose that $\lambda_{G/N}$ is the left Haar measure  normalized with respect to the probability measure $\lambda_N$ of $N$. Then using Proposition \ref{cov.conv.Ncom} we get $\psi\natural \phi=\psi\ast_{G}\phi$. In this case, harmonic analysis on covariant function algebras of characters of compact subgroups studied for different cases in \cite{AGHF.ch.com, AGHF.CJM} and references therein.
\end{remark}

Next result gives a canonical connection of the covariant convolution defined by (\ref{tc.main.central}) with the convolution in $L^1(G)$.

\begin{theorem}\label{N.normal.thm}
Let $G$ be a locally compact group, $N$ be a closed normal subgroup of $G$, and $\xi\in\Gamma(G,N)$.
Suppose $f,g\in\mathcal{C}_c(G)$ are given. Then 
\begin{equation*}
T_\xi(f\ast_Gg)=T_\xi(f)\natural T_\xi(g).
\end{equation*}
\end{theorem}
\begin{proof}
Suppose $x\in G$ is given. 
Applying Theorem 4.1 of \cite{AGHF.ch.normal.mod}, we have
\begin{equation}\label{NTHalt1}
T_\xi(f\ast_G g)(x)
=\int_Gf(y)T_\xi(g)(y^{-1}x)\dd\lambda_G(y).
\end{equation}
Since $\xi\in\Gamma(G,N)$,  Proposition \ref{J.xi.inv.betaN}(2) implies that 
\begin{equation}\label{NTHalt2}
T_\xi(g)(s^{-1}y^{-1}x)=\overline{\xi(s)}T_\xi(g)(y^{-1}x),
\end{equation}
for every $s\in N$ and $y\in G$.   Therefore, using 
Weil's formula (\ref{6}),  (\ref{NTHalt1}) and (\ref{NTHalt2}), we achieve  
\begin{align*}
\int_Gf(y)T_\xi(g)(y^{-1}x)\dd\lambda_G(y)
&=\int_{G/N}\left(\int_Nf(ys)T_\xi(g)((ys)^{-1}x)\dd\lambda_{N}(s)\right)\dd\lambda_{G/N}(yN)
\\&=\int_{G/N}\left(\int_Nf(ys)T_\xi(g)(s^{-1}y^{-1}x)\dd\lambda_{N}(s)\right)\dd\lambda_{G/N}(yN)
\\&=\int_{G/N}\left(\int_Nf(ys)\overline{\xi(s)}T_\xi(g)(y^{-1}x)\dd\lambda_{N}(s)\right)\dd\lambda_{G/N}(yN)
\\&=\int_{G/N}\left(\int_Nf(ys)\overline{\xi(s)}\dd\lambda_{N}(s)\right)T_\xi(g)(y^{-1}x)\dd\lambda_{G/N}(yN)
\\&=\int_{G/N}T_\xi(f)(y)T_\xi(g)(y^{-1}x)\dd\lambda_{G/N}(yN)=T_\xi(f)\natural T_\xi(g)(x).
\end{align*}
\end{proof}

\begin{corollary}
{\it Let $G$ be a locally compact group, $N$ be a compact normal subgroup of $G$, and $\xi\in\Gamma(G,N)$. Suppose $\lambda_N$ is the probability measure on $N$ and $f,g\in\mathcal{C}_c(G)$ are given. Then
\[
T_\xi(f\ast_G g)=T_\xi(f)\ast_G T_\xi(g).
\]
}\end{corollary}

\begin{proposition}\label{cc.bp}
{\it Let $G$ be a locally compact group, $N$ be a closed normal subgroup of $G$, 
and $\xi\in\Gamma(G,N)$. 
The mapping $\natural:\mathcal{C}_\xi(G,N)\times \mathcal{C}_\xi(G,N)\to \mathcal{C}_\xi(G,N)$ given by 
$(\psi,\varphi)\mapsto\psi\natural\varphi$ is a bilinear product.
}\end{proposition}
\begin{remark}
Let $N$ be a closed normal subgroup of $G$ and $\xi\in\Gamma(G,N)$. 
The module action of $\mathcal{C}_c(G)$ on $\mathcal{C}_\xi(G,N)$ given by (4.1) of \cite{AGHF.ch.normal.mod}, satisfies $f\ast_G T_\xi(g)=T_\xi(f)\natural T_\xi(g)$, for every $f,g\in\mathcal{C}_c(G)$.
\end{remark}

We then continue by discussing norm properties of covariant convolutions. 

\begin{theorem}\label{norm.cc}
{\it Let $G$ be a locally compact group and $N$ be a closed normal subgroup of $G$. Suppose $\xi\in\Gamma(G,N)$, $1\le p<\infty$, and $\psi,\phi\in \mathcal{C}_\xi(G,N)$. Then  
\[
\|\psi\natural\phi\|_{(p)}\le\|\psi\|_{(1)}\|\phi\|_{(p)}.
\]
}\end{theorem}
\begin{proof}
Let $1\le p<\infty$ and $\psi,\phi\in \mathcal{C}_\xi(G,N)$. Using Minkowski's integral inequality, we have 
\begin{align*}
\left(\int_{G/N}|\psi\natural\phi(x)|^p\dd\lambda_{G/N}(xN)\right)^{1/p}
&=\left(\int_{G/N}\left|\int_{G/N}\psi(y)\phi(y^{-1}x)\dd\lambda_{G/N}(yN)\right|^p\dd\lambda_{G/N}(xN)\right)^{1/p}
\\&\le\int_{G/N}\left(\int_{G/N}\left|\psi(y)\phi(y^{-1}x)\right|^p\dd\lambda_{G/N}(xN) \right)^{1/p}\dd\lambda_{G/N}(yN)
\\&=\int_{G/N}|\psi(y)|\left(\int_{G/N}\left|\phi(y^{-1}x)\right|^p\dd\lambda_{G/N}(xN) \right)^{1/p}\dd\lambda_{G/N}(yN).
\end{align*}
Let $y\in G$ be given. Since $\lambda_{G/N}$ is $G$-invariant, we get 
\[
\left(\int_{G/N}\left|\phi(y^{-1}x)\right|^p\dd\lambda_{G/N}(xN) \right)^{1/p}=\left(\int_{G/N}\left|\phi(x)\right|^p\dd\lambda_{G/N}(yxN) \right)^{1/p}=\|\phi\|_{(p)}.
\]
Therefore, we obtain  
\begin{align*}
\|\psi\natural\phi\|_{(p)}&\le\int_{G/N}|\psi(y)|\left(\int_{G/N}\left|\phi(y^{-1}x)\right|^p\dd\lambda_{G/N}(xN) \right)^{1/p}\dd\lambda_{G/N}(yN)
\le\|\psi\|_{(1)}\|\phi\|_{(p)}.
\end{align*}
\end{proof}

\begin{corollary}
{\it Let $G$ be a locally compact group, 
$N$ be a closed normal subgroup of $G$, and $\xi\in\Gamma(G,N)$. Then, 
$(\mathcal{C}_\xi(G,N),\natural,\|\cdot\|_{(1)})$ is a normed algebra.
}\end{corollary}
\begin{proof}
Invoking Proposition \ref{cc.bp}, the mapping $\natural:\mathcal{C}_\xi(G,N)\times \mathcal{C}_\xi(G,N)\to \mathcal{C}_\xi(G,N)$ is a bilinear product. Using Theorem \ref{N.normal.thm}, we deduce that 
the bilinear product $\natural$ is associative as well. Then, Theorem \ref{norm.cc} implies that $(\mathcal{C}_\xi(G,N),\natural,\|\cdot\|_{(1)})$ is a normed algebra.
\end{proof}

\begin{proposition}\label{Jxi.hom.C}
{\it Let $G$ be a locally compact group, $N$ be a closed normal subgroup of $G$, and $\xi\in\Gamma(G,N)$. Then,
\begin{enumerate}
\item $T_\xi:\mathcal{C}_c(G)\to \mathcal{C}_\xi(G,N)$
is a contraction homomorphism of normed algebras.
\item $\mathcal{N}_\xi(G,N)$ is an ideal of $\mathcal{C}_c(G)$.
\end{enumerate}
}\end{proposition}

Next result extends the covariant convolution defined by (\ref{tc.main.central}) to the Banach spaces $L^p_\xi(G,N)$. 

\begin{theorem}\label{cc.p}
Let $G$ be a locally compact group and $N$ be a closed normal subgroup of $G$. Suppose $\xi\in\Gamma(G,N)$ and $1<p<\infty$. Then
\begin{enumerate}
\item The covariant convolution  
$\natural:\mathcal{C}_\xi(G,N)\times \mathcal{C}_\xi(G,N)\to \mathcal{C}_\xi(G,N)$ 
 has a unique extension to 
$\natural:L^1_\xi(G,N)\times L^1_\xi(G,N)\to L^1_\xi(G,N)$, 
in which the Banach space $L^1_\xi(G,N)$ equipped with the extended covariant convolution is a Banach algebra.
\item The covariant convolution $\natural:\mathcal{C}_\xi(G,N)\times \mathcal{C}_\xi(G,N)\to \mathcal{C}_\xi(G,N)$ has a unique extension to $\natural:L^1_\xi(G,N)\times L^p_\xi(G,N)\to L^p_\xi(G,N)$, in which the Banach space $L^p_\xi(G,N)$ equipped with the extended covariant convolution is a Banach left $L^1_\xi(G,N)$-module.
\end{enumerate}
\end{theorem}
\begin{proof}
(1) Let $\psi,\phi\in L^1_\xi(G,N)$. Suppose $(\psi_n),(\phi_n)\subset \mathcal{C}_\xi(G,N)$ with $\psi=\lim_n\psi_n$ and $\phi=\lim_n\phi_n$ in $L^1_\xi(G,N)$. Then $(\psi_n\natural\phi_n)$ is a Cauchy sequence in $L^1_\xi(G,N)$. Since $L^1_\xi(G,N)$ is a Banach space, $(\psi_n\natural\phi_n)$ has a limit in $L^1_\xi(G,N)$. We then define 
$\psi\natural\phi:=\lim_n\psi_n\natural\phi_n$. Then $\psi\natural\phi$ is independent of the choice of the sequences $(\psi_n)$ and $(\phi_n)$.
We also have $\psi\natural\phi\in L^1_\xi(G,N)$ and  
$\|\psi\natural\phi\|_{(1)}\le\|\psi\|_{(1)}\|\phi\|_{(1)}$.
Therefore, the Banach space $L^1_\xi(G,N)$ equipped with the extended covariant convolution $\natural:L^1_\xi(G,N)\times L^1_\xi(G,N)\to L^1_\xi(G,N)$ is a Banach algebra.

(2) Let $\psi\in L^1_\xi(G,N)$ and $\phi\in L^p_\xi(G,N)$. Suppose $(\psi_n),(\phi_n)\subset \mathcal{C}_\xi(G,N)$ with $\psi=\lim_n\psi_n$ in $L^1_\xi(G,N)$ and $\phi=\lim_n\phi_n$ in $L^p_\xi(G,N)$. Then $(\psi_n\natural\phi_n)$ is a Cauchy sequence in $L^p_\xi(G,N)$. Since $L^p_\xi(G,N)$ is a Banach space, $(\psi_n\natural\phi_n)$ has a limit in $L^p_\xi(G,N)$. We then define 
$\psi\natural\phi:=\lim_n\psi_n\natural\phi_n$.
Then $\psi\natural\phi\in L^p_\xi(G,N)$ and $\|\psi\natural\phi\|_{(p)}\le\|\psi\|_{(1)}\|\phi\|_{(p)}$. So the Banach space $L^p_\xi(G,N)$ equipped with the extended covariant convolution $\natural:L^1_\xi(G,N)\times L^p_\xi(G,N)\to L^p_\xi(G,N)$ is a Banach left $L^1_\xi(G,N)$-module.
\end{proof}

\begin{proposition}\label{Jxi.hom.L1}
{\it Let $G$ be a locally compact group, $N$ be a closed normal subgroup of $G$, and $\xi\in\Gamma(G,N)$. Then
\begin{enumerate}
\item $T_\xi:L^1(G)\to L^1_\xi(G,N)$
is a contraction homomorphism of Banach algebras.
\item $\mathcal{N}_\xi^1(G,N)$ is a closed ideal of $L^1(G)$. 
\end{enumerate}
}\end{proposition}
\begin{proof}
(1) Theorem 4.4 of \cite{AGHF.ch.normal.main} guarantees contraction property of $T_\xi:L^1(G)\to L^1_\xi(G,N)$. Let $f,g\in L^1(G)$. Suppose $(f_n),(g_n)\subset\mathcal{C}_c(G)$ with $f=\lim_n f_n$ and $g=\lim_n g_n$ in $L^1(G)$. Then $f\ast_G g=\lim_n f_n\ast_G g_n$. Therefore, using continuity of $T_\xi$ and Theorem \ref{N.normal.thm}, we obtain 
\begin{align*}
T_\xi(f\ast_G g)&=T_\xi(\lim_n f_n\ast_G g_n)
=\lim_nT_\xi(f_n\ast_G g_n)
\\&=\lim_n T_\xi(f_n)\natural T_\xi(g_n)=T_\xi(f)\natural T_\xi(g).
\end{align*}
(2) It follows from (1).
\end{proof}

\begin{theorem}
Let $G$ be a locally compact group, $N$ be a closed normal subgroup of $G$, and $\xi\in\Gamma(G,N)$. The covariant Banach function algebra 
$L^1_\xi(G,N)$ possesses a two-sided approximate identity. In particular, if $G$ is second-countable, $L^1_\xi(G,N)$ possesses a two-sided sequential approximate identity.
\end{theorem}
\begin{proof}
Invoking Proposition 3.5.9 of \cite{50}, let $\mathcal{E}:=(f_\alpha)_{\alpha\in\Lambda}$ be a two-sided approximate identity for the Banach algebra $L^1(G)$. Then, for each $f\in L^1(G)$, we have 
\begin{equation}\label{api.G}
f=\lim_\alpha f\ast_G f_\alpha=\lim_\alpha f_\alpha\ast_G f.
\end{equation}
For every $\alpha\in\Lambda$, define $\psi_\alpha:=T_\xi(f_\alpha)$. We then have 
$\mathcal{E}_\xi:=(\psi_\alpha)_{\alpha\in\Lambda}\subset L^1_\xi(G,N)$.
Suppose $\psi\in L^1_\xi(G,N)$ is arbitrary. Let $f\in L^1(G)$ with $T_\xi(f)=\psi$. Using continuity of $T_\xi$, Proposition \ref{Jxi.hom.L1}(1), and (\ref{api.G}), we get 
\begin{align*}
\psi&=T_\xi(\lim_\alpha f\ast_G f_\alpha)
=\lim_\alpha T_\xi(f\ast_G f_\alpha)
\\&=\lim_\alpha T_\xi(f)\natural T_\xi(f_\alpha)
=\lim_\alpha\psi\natural\psi_\alpha.
\end{align*}
The same argument implies that $\lim_\alpha\psi\natural\psi_\alpha=\psi$ as well, which completes the proof. If $G$ is second-countable, we can pick a sequential two-sided approximate identity for the Banach algebra $L^1(G)$ as $\mathcal{E}$. In this case, $\mathcal{E}_\xi$ constitutes a sequential two-sided approximate identity for the covariant Banach algebra $L^1_\xi(G,N)$.
\end{proof}

We then conclude the following explicit construction for the extended covariant convolution given by Theorem \ref{cc.p}.
 
\begin{theorem}
{\it Let $G$ be a locally compact group and $N$ be a closed normal subgroup of $G$. Suppose $\xi\in\Gamma(G,N)$, $1\le p<\infty$, $\phi\in L^1_\xi(G,N)$, and $\psi\in L^p_\xi(G,N)$. Then 
\begin{equation}\label{psi*pshiLpFormula}
\phi\natural\psi(x)=\int_{G/N}\phi(y)\psi(y^{-1}x)\dd\lambda_{G/N}(yN)\ \ \ \ \ {\rm for\ \ l.a.e.}\ x\in G.
\end{equation}
}\end{theorem}
\begin{proof}
Let $\phi\in L^1_\xi(G,N)$ and $\psi\in L^p_\xi(G,N)$. 
Suppose $(\phi_n),(\psi_n)\subset \mathcal{C}_\xi(G,N)$ with $\|\phi_n-\phi\|_{(1)}<2^{-(n+2)}$ and $\|\psi_n-\psi\|_{(p)}<2^{-(n+2)}$ for $n\in\mathbb{N}$. Then $\phi\natural\psi=\lim_n\phi_n\natural\psi_n$ in $L^p_\xi(G,N)$. Suppose $M>0$ with $\|\phi_n\|_{(1)}<M$ and $\|\psi_n\|_{(p)}<M$ for every $n\in\mathbb{N}$. So, $\|\phi_{n+1}\natural\psi_{n+1}-\phi_n\natural\psi_n\|_{(p)}<M2^{-n}$ for $n\in\mathbb{N}$. Therefore, $\phi\natural\psi(x)=\lim_n\phi_n\natural\psi_n(x)$ for l.a.e. $x\in G$, according to Remark 3.6 of \cite{AGHF.ch.normal.mod}. Then    
\begin{align*}
&\left(\int_{G/N}\left(\int_{G/N}|\phi_{n}(y)\psi_{n}(y^{-1}x)-\phi(y)\psi(y^{-1}x)|\dd\lambda_{G/N}(yN)\right)^p\dd\lambda_{G/N}(xN)\right)^{1/p}
\\&\le\int_{G/N}\left(\int_{G/N}|\phi_{n}(y)\psi_{n}(y^{-1}x)-\phi(y)\psi(y^{-1}x)|^p\dd\lambda_{G/N}(xN)\right)^{1/p}\dd\lambda_{G/N}(yN)
\\&\le \|\phi_{n}\|_{(1)}\|\psi_{n}-\psi\|_{(p)}+\|\phi_{n}-\phi\|_{(1)}\|\psi\|_{(p)}<M2^{-(n+1)},
\end{align*}
which implies that  
\[
\int_{G/N}\left(\int_{G/N}|\phi_{n}(y)\psi_{n}(y^{-1}x)-\phi(y)\psi(y^{-1}x)|\dd\lambda_{G/N}(yN)\right)^p\dd\lambda_{G/N}(xN)<M^p2^{-p(n+1)},
\]
for $n\in\mathbb{N}$. Therefore, 
\[
\lim_n\left(\int_{G/N}|\phi_{n}(y)\psi_{n}(y^{-1}x)-\phi(y)\psi(y^{-1}x)|\dd\lambda_{G/N}(yN)\right)^p=0,
\]
for a.e. $xN\in G/N$. Since 
\[
\left|\phi_{n}\natural\psi_{n}(x)-\int_{G/N}\phi(y)\psi(y^{-1}x)\dd\lambda_{G/N}(yN)\right|\le\int_{G/N}|\phi_{n}(y)\psi_{n}(y^{-1}x)-\phi(y)\psi(y^{-1}x)|\dd\lambda_{G/N}(yN),
\]
we get   
\[
\lim_n\left|\phi_{n}\natural\psi_{n}(x)-\int_{G/N}\phi(y)\psi(y^{-1}x)\dd\lambda_{G/N}(yN)\right|=0,
\]
for a.e. $xN\in G/N$. This implies that the right hand side of  (\ref{psi*pshiLpFormula}) is well-defined as a function on $G$, for l.a.e $x\in G$. We also obtain   
\[
\lim_n\phi_{n}\natural\psi_{n}(x)=\int_{G/N}\phi(y)\psi(y^{-1}x)\dd\lambda_{G/N}(yN),
\]
for l.c.a $x\in G$. Then, for l.c.a $x\in G$, we achieve  
\[
\phi\natural\psi(x)=\lim_n\phi_{n}\natural\psi_{n}(x)=\int_{G/N}\phi(y)\psi(y^{-1}x)\dd\lambda_{G/N}(yN).
\]
\end{proof}
\begin{remark}
Let $N$ be a closed normal subgroup of $G$ and $\xi\in\Gamma(G,N)$. 
The module action of $L^1(G)$ on $L^1_\xi(G,N)$ given by (4.3) of \cite{AGHF.ch.normal.mod}, satisfies $f\ast_G T_\xi(g)=T_\xi(f)\natural T_\xi(g)$, for every $f,g\in L^1(G)$.
\end{remark}
\subsection{Covariant Involutions}
We then present the abstract notion of covariant involutions for covariant functions associated to invariant characters of normal subgroups. 
Throughout, we still assume that $G$ is a locally compact group with a left Haar measure $\lambda_G$ and $N$ is a closed normal subgroup of $G$ with a left Haar measure $\lambda_N$. Also, suppose that $\lambda_{G/N}$ is the left Haar measure on the quotient group $G/N$ normalized with respect to Weil's formula (\ref{6}) and $\xi\in\Gamma(G,N)$. 

For $\psi\in \mathcal{C}_\xi(G,N)$, define the {\it covariant involution} 
$\psi^\natural:G\to\mathbb{C}$ via 
\begin{equation}\label{ci.main.normal}
\psi^\natural(x):=\Delta_{G/N}(x^{-1}N)\overline{\psi(x^{-1})}=\frac{\sigma_N(x)}{\Delta_G(x)}\overline{\psi(x^{-1})},
\end{equation}
for all $x\in G$.

\begin{remark}
Let $\xi=1$ be the trivial character of the normal subgroup $N$. Suppose that  
$\psi\in \mathcal{C}_1(G,N)$ is identified with $\widetilde{\psi}\in\mathcal{C}_c(G/N)$. Then covariant involution $\psi^\natural$ coincides with the standard involution of $\widetilde{\psi}$ on $G/N$. 
\end{remark}
\begin{remark}
Let $N$ be a compact normal subgroup of $G$. Then $\sigma_N=1$.
Therefore, for every $\xi\in\Gamma(G,N)$ and $\psi\in \mathcal{C}_\xi(G,N)\subseteq\mathcal{C}_c(G)$ we get $\psi^\natural=\psi^{\ast^G}$.
\end{remark}

\begin{remark}\label{tws.inv}
The covariant involution (\ref{ci.main.normal}) gives a universal characterization of twisted involutions on $G/N$ associated to certain class of multipliers,  if there is a Borel cross-section for $G/N$ in $G$, see \cite{MackActa1958, LepMAn1965,LepInv1-1967, LepInv2-1967,LepInv1968,BuSmTrans1970}.   Suppose that $\mathrm{s}:G/N\to G$ is a Borel cross-section for $G/N$ in $G$ and $\mathrm{r}:G\to N$ is given by $\mathrm{r}(x):=\mathrm{s}(xN)^{-1}x$, for every $x\in G$.  Assume that the multiplier $\beta:G/N\times G/N\to\mathbb{T}$ is given by 
\[
\beta(xN,yN):=\overline{\xi(\mathrm{s}(xyN)^{-1}\mathrm{s}(xN)\mathrm{s}(yN))},
\]
for every $xN,yN\in G/N$. Then 
\[
\psi^\natural(x)=\xi(\mathrm{r}(x))\psi_\mathrm{s}^{\ast_\beta}(xN),
\]
for every $\psi\in\mathcal{C}_\xi(G,N)$ and $x\in G$, where $\psi_\mathrm{s}:=\psi\circ\mathrm{s}$ and $^{\ast_\beta}$ is the twisted involution associated to $\beta$, see \cite{MackActa1958, LepMAn1965,LepInv1-1967, LepInv2-1967,LepInv1968,BuSmTrans1970}.
\end{remark}

Next result gives a canonical connection of the covariant involution 
defined by (\ref{ci.main.normal}) with the involution 
in $L^1(G)$ via the linear operator $T_\xi$.

\begin{theorem}\label{N.normal.thm.cinv}
Let $G$ be a locally compact group and $N$ be a closed normal subgroup of $G$. Suppose $\xi\in\Gamma(G,N)$ and $f\in\mathcal{C}_c(G)$. Then 
\begin{equation*}
T_\xi(f^{\ast^G})=T_\xi(f)^\natural.
\end{equation*}
\end{theorem}
\begin{proof}
Let $f\in\mathcal{C}_c(G)$ and $x\in G$. We then have 
\begin{align*}
T_\xi(f^{\ast^G})(x)&=\int_Nf^{*_G}(xs)\overline{\xi(s)}\dd\lambda_N(s)
=\int_N\overline{f((xs)^{-1})}\Delta_G((xs)^{-1})\overline{\xi(s)}\dd\lambda_N(s)
\\&=\int_N\overline{f(s^{-1}x^{-1})}\Delta_G(s^{-1}x^{-1})\overline{\xi(s)}\dd\lambda_N(s)=\Delta_G(x^{-1})\int_N\overline{f(s^{-1}x^{-1})}\Delta_G(s^{-1})\overline{\xi(s)}\dd\lambda_N(s).
\end{align*}
Replacing $s$ to $s^{-1}$, we get 
\begin{align*}
T_\xi(f^{\ast^G})(x)&=\Delta_G(x^{-1})\int_N\overline{f(s^{-1}x^{-1})}\Delta_G(s^{-1})\overline{\xi(s)}\dd\lambda_N(s)
\\&=\Delta_G(x^{-1})\int_N\overline{f(sx^{-1})}\Delta_G(s)\xi(s)\dd\lambda_N(s^{-1})
\\&=\Delta_G(x^{-1})\int_N\overline{f(sx^{-1})}\Delta_G(s)\Delta_N(s^{-1})\xi(s)\dd\lambda_N(s).
\end{align*}
Since $N$ is normal in $G$, we have $\Delta_G|_N=\Delta_N$, see Proposition 3.3.17 of \cite{50}. Therefore,  
\begin{align*}
T_\xi(f^{\ast^G})(x)&=\Delta_G(x^{-1})\int_N\overline{f(sx^{-1})}\Delta_G(s)\Delta_N(s^{-1})\xi(s)\dd\lambda_N(s)
\\&=\Delta_G(x^{-1})\int_N\overline{f(sx^{-1})}\Delta_N(s)\Delta_N(s^{-1})\xi(s)\dd\lambda_N(s)=\Delta_G(x^{-1})\int_N\overline{f(sx^{-1})}\xi(s)\dd\lambda_N(s).
\end{align*}
Since $\xi$ is $G$-invariant, using Proposition \ref{J.xi.inv.betaN}(1), we get 
\begin{align*}
T_\xi(f^{\ast^G})(x)&=\Delta_G(x^{-1})\int_N\overline{f(sx^{-1})}\xi(s)\dd\lambda_N(s)
\\&=\Delta_G(x^{-1})\sigma_N(x)\int_N\overline{f(x^{-1}s)}\xi(s)\dd\lambda_N(s)=\Delta_G(x^{-1})\sigma_N(x)\overline{T_\xi(f)(x^{-1})}=T_\xi(f)^\natural(x).
\end{align*}
\end{proof}

\begin{corollary}
{\it Let $G$ be a locally compact group and $N$ be a compact normal subgroup of $G$. Suppose $\xi\in\Gamma(G,N)$ and $f\in\mathcal{C}_c(G)$. Then 
\begin{equation*}
T_\xi(f^{\ast^G})=T_\xi(f)^{\ast^G}.
\end{equation*}
}\end{corollary}

We then have the following observations concerning covariant involutions. 

\begin{proposition}\label{main.prop.cinv}
{\it Let $G$ be a locally compact group, $N$ be a closed normal subgroup of $G$, and $\xi\in\Gamma(G,N)$. Suppose $\psi,\phi\in \mathcal{C}_\xi(G,N)$ are given. Then  
\begin{enumerate}
\item $\|\psi^\natural\|_{(1)}=\|\psi\|_{(1)}$.
\item $\psi^{{\natural}^{\natural}}=\psi$.
\item $(\psi\natural\phi)^\natural=\phi^\natural\natural\psi^\natural$.
\end{enumerate}
}\end{proposition}
\begin{proof}
(1) We have 
\begin{align*}
\|\psi^\natural\|_{(1)}&=\int_{G/N}|\psi^\natural(x)|\dd\lambda_{G/N}(xN)=\int_{G/N}\Delta_{G/N}(x^{-1}N)|\psi(x^{-1})|\dd\lambda_{G/N}(xN)
\\&=\int_{G/N}\Delta_{G/N}(xN)|\psi(x)|\dd\lambda_{G/N}(x^{-1}N)=\int_{G/N}|\psi(x)|\dd\lambda_{G/N}(xN)=\|\psi\|_{(1)}.
\end{align*} 
(2) It follows from Theorem \ref{N.normal.thm.cinv}.\\
(3) Let $f,g\in\mathcal{C}_c(G)$ with $T_\xi(f)=\psi$ and $T_\xi(g)=\phi$. Using Theorem \ref{N.normal.thm} and Theorem \ref{N.normal.thm.cinv}, we get 
\begin{align*}
(\psi\natural\phi)^\natural
&=T_\xi(f\ast_G g)^\natural
=T_\xi((f\ast_G g)^{*^G})
\\&=T_\xi(g^{*^G}\ast_Gf^{\ast^G})
=T_\xi(g^{*^G})\natural T_\xi(f^{\ast^G})
=T_\xi(g)^\natural\natural T_\xi(f)^\natural=\phi^\natural\natural\psi^\natural.
\end{align*}
\end{proof}

Therefore, the mapping $^{\natural}:(\mathcal{C}_\xi(G,N),\|\cdot\|_{(1)})\to(\mathcal{C}_\xi(G,N),\|\cdot\|_{(1)})$ given by $\psi\mapsto\psi^\natural$ is an isometric involution and hence $(\mathcal{C}_\xi(G,N),\natural,^{\natural},\|\cdot\|_{(1)})$ is a normed $*$-algebra.

We then extend the involution map (\ref{ci.main.normal}) into $L^1_\xi(G,N)$.

\begin{theorem}\label{ci.L1}
Let $G$ be a locally compact group, $N$ be a closed normal subgroup of $G$, and $\xi\in\Gamma(G,N)$.  The covariant involution  
$^\natural:\mathcal{C}_\xi(G,N)\to \mathcal{C}_\xi(G,N)$ has a unique extension to an isometric involution $^\natural:L^1_\xi(G,N)\to L^1_\xi(G,N)$, 
in which the covariant Banach algebra $L^1_\xi(G,N)$ equipped with the extended involution is a Banach $*$-algebra.
\end{theorem}
\begin{proof}
Let $\psi\in L^1_\xi(G,N)$. Suppose $(\psi_n)\subset \mathcal{C}_\xi(G,N)$ with $\psi=\lim_n\psi_n$ in $L^1_\xi(G,N)$. Then $(\psi_n^\natural)$ is Cauchy in $L^1_\xi(G,N)$. Since $L^1_\xi(G,N)$ is a Banach space, $(\psi_n^\natural)$ has a limit in $L^1_\xi(G,N)$. We then define 
$\psi^\natural:=\lim_n\psi_n^\natural$. Then $\psi^\natural$ is independent of the choice of the sequences $(\psi_n)$. Also, we have $\psi^\natural\in L^1_\xi(G,N)$, $\psi^{{\natural}^{\natural}}=\psi$, and $\|\psi^\natural\|_{(1)}=\|\psi\|_{(1)}$. Further, let $\phi\in L^1_\xi(G,N)$ and $(\phi_n)\subset \mathcal{C}_\xi(G,N)$ with $\phi=\lim_n\phi_n$ in $L^1_\xi(G,N)$. Then, using structure of covariant convolution on $L^1_\xi(G,N)$ and Proposition \ref{main.prop.cinv}(3), we get 
\[
(\psi\natural\phi)^\natural=\lim_n(\psi_n\natural\phi_n)^\natural=\lim_n\phi_n^\natural\natural\psi_n^\natural=\phi^\natural\natural\psi^\natural.
\]
Therefore, the covariant Banach algebra $L^1_\xi(G,N)$ equipped with the extended covariant involution $^\natural:L^1_\xi(G,N)\to L^1_\xi(G,N)$ is a Banach $*$-algebra.
\end{proof}

\begin{proposition}
{\it Let $G$ be a locally compact group, $N$ be a closed normal subgroup of $G$, and $\xi\in\Gamma(G,N)$. Then 
\begin{enumerate}
\item $T_\xi:L^1(G)\to L^1_\xi(G,N)$ 
is a contraction $*$-homomorphism of Banach $*$-algebras.
\item $\mathcal{N}_\xi^1(G,N)$ is a closed $*$-ideal of $L^1(G)$.
\item The Banach $*$-algebra $L_\xi^1(G,N)$ is isometrically isomorphic to the quotient Banach $*$-algebra $L^1(G)/\mathcal{N}_\xi^1(G,N)$.
\end{enumerate}
}\end{proposition}
\begin{proof}
(1) Suppose $f\in L^1(G)$ and $(f_n)\subset\mathcal{C}_c(G)$ with $f=\lim_n f_n$ in $L^1(G)$. Then $f^{\ast^G}=\lim_n f_n^{\ast^G}$ in $L^1(G)$. 
Using continuity of the linear map $T_\xi:L^1(G)\to L^1_\xi(G,N)$ and Theorem \ref{N.normal.thm.cinv}, we achieve 
\begin{align*}
T_\xi(f^{\ast^G})&=T_\xi(\lim_n f_n^{\ast^G})
=\lim_n T_\xi(f_n^{\ast^G})=\lim_n T_\xi(f_n)^\natural=T_\xi(f)^\natural.
\end{align*}
Then Proposition \ref{Jxi.hom.L1}(1) guarantees that $T_\xi:L^1(G)\to L^1_\xi(G,N)$ is a contraction $*$-homomorphism of Banach $*$-algebras.\\
(2-3) follow from (1).
\end{proof}

\begin{corollary}
{\it Let $G$ be a locally compact group, $N$ be a compact normal subgroup of $G$, and $\xi\in\Gamma(G,N)$. Then, $(L^1_\xi(G,N),\ast_G,^{\ast^G},\|\cdot\|_{L^1(G)})$ is a Banach $*$-algebra. 
}\end{corollary} 

We finish this section by the following results concerning compact groups. 

\begin{theorem}
Let $G$ be a compact group and $N$ be a closed normal subgroup of $G$.
Suppose $\xi\in\Gamma(G,N)$, and $1<p<\infty$. Then 
\begin{enumerate}
\item $(L^p_\xi(G,N),\ast_G,^{\ast^G},\|\cdot\|_{L^p(G)})$ is a Banach $*$-algebra. 
\item $T_\xi:L^p(G)\to L^p_\xi(G,N)$
is a contraction $*$-homomorphism of Banach $*$-algebras.
\item  $\mathcal{N}_\xi^p(G,N)$ is a closed $*$-ideal of $L^p(G)$.
\item The Banach $*$-algebra $L_\xi^p(G,N)$ is isometrically isomorphic to the quotient Banach $*$-algebra $L^p(G)/\mathcal{N}_\xi^p(G,N)$.
\end{enumerate}
\end{theorem}
\begin{proof}
(1) Suppose that $G$ is a compact group and $N$ is a closed normal subgroup of $G$. Then $N$ is compact. So using (\ref{p.comp}), we obtain $L^p_\xi(G,N)\subseteq L^p(G)$. Since $G$ is compact, $(L^p(G),\ast_G,^{\ast^G})$ is a Banach $*$-algebra. Then $L^p_\xi(G,N)$ is a closed $*$-ideal of $L^p(G)$. Therefore, we conclude that $(L^p_\xi(G,N),\ast_G,^{\ast^G},\|\cdot\|_{L^p(G)})$ is a Banach $*$-algebra.\\
(2-4) follow from (1).   
\end{proof}

\begin{remark}
The covariant function $*$-algebra $(L^1_\xi(G,N),\natural,^\natural)$
gives a universal isometrically isomorphic characterization for twisted function $*$-algebras on $G/N$ associated to certain class of multipliers,  if there is a Borel cross-section for $G/N$ in $G$. 
Let $\mathrm{s}:G/N\to G$ be a Borel cross-section for $G/N$ in $G$.  Suppose that $\mathrm{r}:G\to N$ is given by $\mathrm{r}(x):=\mathrm{s}(xN)^{-1}x$, for every $x\in G$.  Assume that the multiplier $\beta:G/N\times G/N\to\mathbb{T}$ is given by 
\[
\beta(xN,yN):=\overline{\xi(\mathrm{s}(xyN)^{-1}\mathrm{s}(xN)\mathrm{s}(yN))},
\]
for every $xN,yN\in G/N$.  Then Remarks \ref{tws.cnv} and \ref{tws.inv}, implies that the covariant function $*$-algebra $(L^1_\xi(G,N),\natural,^\natural)$ is isometrically isomorphic to the twisted function $*$-algebra $(L^1(G/N),\ast_\beta,^{\ast^\beta})$, see \cite{MackActa1958, LepMAn1965,LepInv1-1967, LepInv2-1967,LepInv1968,BuSmTrans1970} for more details on structure of twisted function $*$-algebras.    
\end{remark}

\section{\bf Covariant Function Algebras on Semi-direct Product Groups}

Throughout this section, we investigate the theory of covariant convolutions/involutions for characters of canonical normal subgroups in  
semi-direct product of locally compact groups. 
We here assume that $H,K$ are locally compact groups and $\theta:H\to Aut(K)$ is a continuous homomorphism. Suppose $G_\theta=H\ltimes_\theta K$ is the semi-direct product of $H$ and $K$ with respect to $\theta$. 
The semi-direct product $G_\theta=H\ltimes_\theta K$ is the locally compact topological group with the underlying set $H\times K$
which is equipped by the product topology and the group operation 
\begin{equation*}
(h,k)\ltimes_\theta(h',k')=(hh',k\theta_h(k'))\hspace{0.5cm}{\rm and}\hspace{0.5cm}(h,k)^{-1}=(h^{-1},\theta_{h^{-1}}(k^{-1})).
\end{equation*}

Let $N$ be a closed normal subgroup of $K$ such that $\theta_h(N)=N$, for every $h\in H$. In this case, $N$ is called a $\theta$-invariant subgroup of $K$. For each $h\in H$, assume that $\delta_{H,N}(h)\in(0,\infty)$ is given by 
\begin{equation}\label{delta.H.N}
\dd\lambda_N(\theta_h(s))=\delta_{H,N}(h)^{-1}\dd\lambda_N(s),
\end{equation}
where $\lambda_N$ is a left Haar measure on $N$.
Then the map $\delta_{H,N}:H\to(0,\infty)$ defined by $h\mapsto\delta_{H,N}(h)$ is a continuous homomorphism. 

The continuous homomorphism $\delta_{H,K}$ characterizes Haar measure and modular function on the semi-direct product $G_\theta:=H\ltimes_\theta K$.
Let $\Delta_H$ and $\Delta_K$ be modular functions of $H$ and $K$ respectively. The modular function $\Delta_{G_\theta}$ of $G_\theta$ is given by $\Delta_{G_\theta}(h,k)=\delta_{H,K}(h)\Delta_H(h)\Delta_K(k)$.
Suppose that $\lambda_H$ and $\lambda_K$ are left Haar measures on $H$ and $K$ respectively. Then 
\begin{equation}\label{lam.theta}
\dd\lambda_{G_\theta}(h,k):=\delta_{H,K}(h)\dd\lambda_H(h)\dd\lambda_K(k),
\end{equation}
is a left Haar measure on $G_\theta$, see \cite[\S15.26 and \S15.29]{HR1}.

\begin{proposition}
{\it Let $H,K$ be locally compact groups and $\theta:H\to Aut(K)$ be a continuous homomorphism. Suppose $G_\theta=H\ltimes_\theta K$, $N$ is a closed normal $\theta$-invariant subgroup of $K$, and $(h,k)\in G_\theta$. Then 
\begin{enumerate}
\item $\sigma_N(h)=\delta_{H,N}(h)$.
\item If $k\in N$ then $\sigma_N(h,k)=\delta_{H,N}(h)\Delta_N(k)$.
\end{enumerate}
}\end{proposition}
\begin{proof}
(1) Let $x:=(h,e_K)\in G_\theta$. Then, for each $s\in N$, we have 
\[
x^{-1}\ltimes(e_H,s)\ltimes x=(e_H,\theta_{h^{-1}}(s)).
\]
Suppose $\lambda_N$ is a fixed left Haar measure on $N$. Using (\ref{tau.main}) and (\ref{delta.H.N}), for each $v\in\mathcal{C}_c(N)$, we get 
\begin{align*}
\sigma_N(h)\int_Nv(s)\dd\lambda_N(s)
&=\int_Nv(s)\dd\lambda_N(x^{-1}\ltimes(e_H,s)\ltimes x)
\\&=\int_Nv(s)\dd\lambda_N(\theta_{h^{-1}}(s))
=\delta_{H,N}(h)\int_Nv(s)\dd\lambda_N(s),
\end{align*}
which implies that $\sigma_N(h)=\delta_{H,N}(h)$.\\
(2) Suppose $k\in N$. Using (1) and since $\sigma_N:G_\theta\to(0,\infty)$ is a group homomorphism, we obtain 
\[
\sigma_N(h,k)=\sigma_N((e_H,k)\ltimes(h,e_K))=\sigma_N(k)\sigma_N(h)=\Delta_N(k)\delta_{H,N}(h).
\] 
\end{proof}
\begin{corollary}
{\it Let $H,K$ be locally compact groups and $\theta:H\to Aut(K)$ be a continuous homomorphism. Suppose $G_\theta=H\ltimes_\theta K$ and $(h,k)\in G_\theta$. Then 
\[
\sigma_K(h,k)=\delta_{H,K}(h)\Delta_K(k).
\] 
}\end{corollary}
The continuous homomorphism $\theta:H\to Aut(K)$ induces the continuous homomorphism $\widetilde{\theta}:H\to Aut(K/N)$, given by $h\mapsto\widetilde{\theta}_h$, where $\widetilde{\theta}_h:K/N\to K/N$ is defined by \[
\widetilde{\theta}_h(kN):=\theta_h(k)N,
\] 
for every $h\in H$ and $kN\in K/N$. 

Then the quotient group $G_\theta/N$ is canonically isomorphic as a topological group to the semi-direct product $G_{\widetilde{\theta}}:=H\ltimes_{\widetilde{\theta}}(K/N)$ via the topological group isomorphism $(h,k)N\mapsto (h,kN)$.  Invoking Proposition 5.1(1) of \cite{AGHF.ch.normal.mod}, we have 
$\delta_{H,K}(h)=\delta_{H,N}(h)\delta_{H,K/N}(h)$ for every $h\in H$. Therefore, for $(h,k)\in G_\theta$, we obtain  
\begin{equation}\label{DeGN}
\Delta_{G_\theta/N}((h,k)N)=\delta_{H,K}(h)\frac{\Delta_H(h)\Delta_K(k)}{\sigma_N(h,k)}=\delta_{H,K/N}(h)\frac{\Delta_H(h)\Delta_K(k)}{\sigma_N(k)}.
\end{equation}
Assume that $\lambda_{K/N}$ is the left Haar measure on $K/N$ normalized with respect to left Haar measures $\lambda_N$ on $N$ and $\lambda_K$ on $K$. Then  
\[
\dd\lambda_{G_\theta/N}(h,kN):=\delta_{H,K/N}(h)\dd\lambda_{H}(h)\dd\lambda_{K/N}(kN),
\]
uniquely identifies the left Haar measure on $G_\theta/N$ normalized with respect to the left Haar measure $\lambda_{G_\theta}$ on $G_\theta$ given by (\ref{lam.theta}) and the left Haar measure $\lambda_N$ on $N$, see Proposition 5.1(2) of \cite{AGHF.ch.normal.mod}.

Suppose that $N$ is a closed subgroup of $K$ such that $\theta_h(N)=N$, for every $h\in H$. Assume that $\lambda_N$ is a fixed left Haar measure on $N$ and $\xi\in\chi(N)$. For $f\in\mathcal{C}_c(G_\theta)$, the function $T_\xi(f):G_\theta\to\mathbb{C}$ is given by  
\[
T_\xi(f)(a,b)=\int_Nf(a,b\theta_a(s))\overline{\xi(s)}\dd\lambda_N(s),
\]
for all $(a,b)\in G_\theta$. Then, using (\ref{delta.H.N}), we have 
\begin{equation}\label{Jxi.normal.semi}
T_\xi(f)(a,b)=\delta_{H,N}(a)\int_Nf(a,bs)\overline{\xi(\theta_{a^{-1}}(s))}\dd\lambda_N(s).
\end{equation}
In particular, if $b\in N$, we get   
\[
T_\xi(f)(a,b)=\delta_{H,N}(a)\xi(\theta_{a^{-1}}(b))\int_Nf(a,s)\overline{\xi(\theta_{a^{-1}}(s))}\dd\lambda_N(s).
\]
Let $\psi\in \mathcal{C}_\xi(G_\theta,N)$. Then $
\psi(h,k\theta_h(s))=\xi(s)\psi(h,k)$, for $h\in H$, $k\in K$, and $s\in N$. Therefore, 
\begin{equation*}
\psi(h,s)=\xi\circ\theta_{h^{-1}}(s)\psi(h,e_K)\ \ \ \ \ \ {\rm for}\ \ (h,s)\in H\times N.
\end{equation*}
In this case, for $1\le p<\infty$, the norm $\|\psi\|_{(p)}$ can be computed by  
\[
\|\psi\|_{(p)}^p:=\int_H\int_{K/N}|\psi(h,k)|^p\delta_{H,K/N}(h)\dd\lambda_{H}(h)\dd\lambda_{K/N}(kN).
\]
We then have the following characterization for invariant characters of normal subgroups of $K$. 

\begin{proposition}\label{Ginv.Semi}
{\it Let $H,K$ be locally compact groups and $\theta:H\to Aut(K)$ be a continuous homomorphism. Suppose $G_\theta=H\ltimes_\theta K$, $N$ is a closed normal $\theta$-invariant subgroup of $K$, and $\xi\in\chi(N)$. Then, $\xi\in\Gamma(G_\theta,N)$ if and only if 
$\xi(s)=\xi\circ\theta_h(k^{-1}sk)$, for $h\in H$, $k\in K$, and $s\in N$. 
}\end{proposition}
\begin{proof}
Suppose $h\in H$, $k\in K$, and $s\in N$. Let $x:=(h,k)$. We then have
\begin{align*}
x^{-1}\ltimes(e_H,s)\ltimes x
&=(h,k)^{-1}\ltimes(e_H,s)\ltimes(h,k)
=(h^{-1},\theta_{h^{-1}}(k^{-1}))\ltimes(e_H,s)\ltimes(h,k)
\\&=(h^{-1},\theta_{h^{-1}}(k^{-1}s))\ltimes(h,k)
=(h^{-1}h,\theta_{h^{-1}}(k^{-1}sk))=(e_H,\theta_{h^{-1}}(k^{-1}sk)).
\end{align*}  
Therefore, we get 
\[
\xi_x(s)=\xi(x^{-1}\ltimes(e_H,s)\ltimes x)=\xi(\theta_{h^{-1}}(k^{-1}sk)),
\]
implying $\xi\in\Gamma(G_\theta,N)$ if and only if 
$\xi(s)=\xi\circ\theta_h(k^{-1}sk)$, for $h\in H$, $k\in K$, and $s\in N$. 
\end{proof}

\begin{corollary}\label{Ginv.Semi.Kabelian}
{\it Let $H,K$ be locally compact groups and $\theta:H\to Aut(K)$ be a continuous homomorphism. Suppose $G_\theta=H\ltimes_\theta K$, $N$ is a central closed $\theta$-invariant subgroup of $K$, and $\xi\in\chi(N)$. Then, $\xi\in\Gamma(G_\theta,N)$ if and only if $\xi=\xi\circ\theta_h$, for every $h\in H$. In particular, if $K$ is Abelian then $\xi\in\Gamma(G_\theta,N)$ if and only if $\xi=\xi\circ\theta_h$, for every $h\in H$. 
}\end{corollary}

Let $N$ be a closed $\theta$-invariant normal subgroup of $K$, and $\xi\in\Gamma(G_\theta,N)$. Suppose that $f\in\mathcal{C}_c(G_\theta)$ and $(a,b)\in G_\theta$ are given. Then, using (\ref{Jxi.normal.semi}), we obtain  
\begin{equation}\label{Jxi.N.Main}
T_\xi(f)(a,b)=\delta_{H,N}(a)\int_Nf(a,bs)\overline{\xi(s)}\dd\lambda_N(s).
\end{equation}
In particular, if $b\in N$, we get 
\[
T_\xi(f)(a,b)=\delta_{H,N}(a)\xi(b)\int_Nf(a,s)\overline{\xi(s)}\dd\lambda_N(s).
\]
Assume that $\psi,\phi\in \mathcal{C}_\xi(G_\theta,N)$. Then  
\[
\psi(h,ks)=\xi(s)\psi(h,k),
\]
for every $h\in H$, $k\in K$, $s\in N$. The covariant convolution  
$\psi\natural\phi:G_\theta\to\mathbb{C}$ is given by  
\begin{equation}\label{cc.main.semi}
\psi\natural\phi(a,b)=\int_H\int_{K/N}\psi(h,k)\phi(h^{-1}a,\theta_{h^{-1}}(k^{-1}b))\delta_{H,K/N}(h)\dd\lambda_{H}(h)\dd\lambda_{K/N}(kN),
\end{equation}
and the covariant involution $\psi^\natural:G_\theta\to\mathbb{C}$ is given by 
\begin{equation}\label{ci.main.semi}
\psi^\natural(a,b):=\Delta_{G_\theta/N}(a^{-1},\theta_{a^{-1}}(b^{-1})N)\overline{\psi(a^{-1},\theta_{a^{-1}}(b^{-1}))},
\end{equation}
for every $(a,b)\in G_\theta$. 

\subsection{\bf The Case $N=K$}
Let $H,K$ be locally compact groups and $\theta:H\to Aut(K)$ be a continuous homomorphism. Then $\widetilde{K}=\{(e_H,k):k\in K\}$ is a closed normal subgroup of $G_\theta=H\ltimes_\theta K$. 
Suppose that $\xi\in\chi(K)$ and $\lambda_K$ is a left Haar measure on $K$.
Assume that $f\in\mathcal{C}_c(G_\theta)$ and $(a,b)\in G_\theta$. Then, using (\ref{Jxi.normal.semi}), we get 
\[
T_\xi(f)(a,b)=\delta_{H,K}(a)\xi(\theta_{a^{-1}}(b))\int_Kf(a,s)\overline{\xi(\theta_{a^{-1}}(s))}\dd\lambda_K(s).
\] 

Next we present a criterion for $G_\theta$-invariant characters of $K$.

\begin{proposition}\label{Ginv.semi.K}
{\it Let $H,K$ be locally compact groups and $\theta:H\to Aut(K)$ be a continuous homomorphism. Suppose $G_\theta=H\ltimes_\theta K$ and $\xi\in\chi(K)$. Then, $\xi\in\Gamma(G_\theta,K)$ 
if and only if $\xi=\xi\circ\theta_h$ for every $h\in H$.
}\end{proposition}
\begin{proof}
Let $\xi\in \Gamma(G_\theta,K)$ be given. Using Proposition \ref{Ginv.Semi}, we get $\xi(s)=\xi(\theta_a(b^{-1}sb))$,  
for all $a\in H$, $b,s\in K$. If $b:=e_K$ then $\xi(s)=\xi(\theta_a(s))=\xi\circ\theta_a(s)$, for every $a\in H$ and $s\in N$. So $\xi=\xi\circ\theta_h$ for all $h\in H$. Conversely, let $\xi\circ\theta_a=\xi$ for all $a\in H$. We then have  
\begin{align*}
\xi(\theta_a(b^{-1}sb))
&=\xi\circ\theta_a(b^{-1}sb)
=\xi\circ\theta_a(b^{-1})\xi\circ\theta_a(s)\xi\circ\theta_a(b)
\\&=\overline{\xi\circ\theta_a(b)}\xi\circ\theta_a(s)\xi\circ\theta_a(b)=\xi\circ\theta_a(s)=\xi(s),
\end{align*}
for all $a\in H$, $b,s\in K$. Then, applying Proposition \ref{Ginv.Semi}, we get $\xi\in\Gamma(G_\theta,K)$.
\end{proof}

For $\xi\in\Gamma(G_\theta,K)$, $f\in\mathcal{C}_c(G_\theta)$, and $(a,b)\in G_\theta$ we obtain  
\begin{equation}\label{Jxi.K}
T_\xi(f)(a,b)=\delta_{H,K}(a)\xi(b)\int_Kf(a,s)\overline{\xi(s)}\dd\lambda_K(s).
\end{equation}
In this case, every $\psi\in \mathcal{C}_\xi(G_\theta,K)$ satisfies  
\begin{equation}\label{covproperty.K}
\phi(h,k)=\xi(k)\phi(h,e_K),
\end{equation} 
for all $h\in H$ and $k\in K$. Also, for $1\le p<\infty$, we have   
\[
\|\psi\|_{(p)}^p=\int_H|\psi(h,e_K)|^p\dd\lambda_{H}(h).
\]
We finish this section by the following explicit formulas for covariant convolutions and covariant involutions in $\mathcal{C}_\xi(G_\theta,K)$, if $\xi\in\Gamma(G_\theta, K)$.

\begin{proposition}
{\it Let $H,K$ be locally compact groups and $\theta:H\to Aut(K)$ be a continuous homomorphism. Suppose $G_\theta:=H\ltimes_\theta K$, 
$\xi\in\Gamma(G_\theta,K)$, $\psi,\phi\in \mathcal{C}_\xi(G_\theta,K)$, and $(a,b)\in G_\theta$. Then
\begin{enumerate}
\item The covariant convolution  
$\psi\natural\phi:G_\theta\to\mathbb{C}$ is given by  
\begin{equation}\label{cc.main.semi.canoK}
\psi\natural\phi(a,b)=\xi(b)\int_H\psi(h,e_K)\phi(h^{-1}a,e_K)\dd\lambda_{H}(h).
\end{equation}
\item The covariant involution $\psi^\natural:G_\theta\to\mathbb{C}$ is given by 
\begin{equation}\label{ci.main.semi.canoK}
\psi^\natural(a,b)=\xi(b)\Delta_{H}(a^{-1})\overline{\psi(a^{-1},e_K)}.
\end{equation}
\end{enumerate}
}\end{proposition}
\begin{proof}
(1) Using (\ref{cc.main.semi}) and (\ref{covproperty.K}), for $N:=K$, we get 
\begin{align*}
\psi\natural\phi(a,b)&=\int_H\int_{K/N}\psi(h,k)\phi(h^{-1}a,\theta_{h^{-1}}(k^{-1}b))\delta_{H,K/N}(h)\dd\lambda_{H}(h)\dd\lambda_{K/N}(kN)
\\&=\int_H\psi(h,e_K)\phi(h^{-1}a,\theta_{h^{-1}}(b))\dd\lambda_{H}(h)
\\&=\int_H\psi(h,e_K)\phi(h^{-1}a,e_K)\xi(\theta_{h^{-1}}(b))\dd\lambda_{H}(h)
=\xi(b)\int_H\psi(h,e_K)\phi(h^{-1}a,e_K)\dd\lambda_{H}(h).
\end{align*}
(2) Using (\ref{DeGN}), (\ref{ci.main.semi}) and (\ref{covproperty.K}), we have 
\begin{align*}
\psi^\natural(a,b)&=\Delta_{G_\theta/K}(a^{-1},\theta_{a^{-1}}(b^{-1})K)\overline{\psi(a^{-1},\theta_{a^{-1}}(b^{-1}))}
=\Delta_{H}(a^{-1})\overline{\psi(a^{-1},\theta_{a^{-1}}(b^{-1}))}
\\&=\Delta_{H}(a^{-1})\overline{\xi(\theta_{a^{-1}}(b^{-1}))}\overline{\psi(a^{-1},e_K)}=\xi(b)\Delta_{H}(a^{-1})\overline{\psi(a^{-1},e_K)}.
\end{align*}
\end{proof}

\section{\bf Examples and Conclusions}
Throughout this section, we discuss the theory of covariant convolutions and involutions in the case of different examples of semi-direct product groups. 

\subsection{Abstract Weyl-Heisenberg Groups}

Let $L$ be a locally compact Abelian group with the dual group $\widehat{L}$.
Suppose $H:=L$ and $K:=\widehat{L}\times\mathbb{T}$. Let $\theta:H\to Aut(K)$ be the continuous homomorphism given by $x\mapsto\theta_x$, where $\theta_x:K\to K$ is defined by $\theta_x(\omega,z):=(\omega,\omega(x)z)$,
for every $(\omega,z)\in K=\widehat{L}\times\mathbb{T}$.
The abstract Weyl-Heisenberg group associated to the LCA group $L$, denoted by $\mathbb{W}(L)$, is the semi-direct product $H\ltimes_\theta K$. 
The abstract Weyl-Heisenberg group $\mathbb{W}(L)$ has the underlying set $L\times\widehat{L}\times\mathbb{T}$ equipped with the following group law
\[
(x,\omega,z)\ltimes(x',\omega',z')=(x+x',\omega\omega',\omega'(x)zz'),
\]
\[
(x,\omega,z)^{-1}=(-x,\overline{\omega},\omega(x)\overline{z}),
\]
for every $(x,\omega,z),(x',\omega',z')\in \mathbb{W}(L)$.

\subsubsection{\bf The Case $N=\mathbb{T}$}
The set $N:=\{(e_L,\mathbf{1},z):z\in\mathbb{T}\}$ is a closed and central subgroup of 
$\mathbb{W}(L)$. Also, we have $\theta_x(\mathbf{1},z)=(\mathbf{1},z)$ for all $x\in L$ and $z\in\mathbb{T}$. We then have $\chi(N)=\{\chi_n:n\in \mathbb{Z}\}$, where the character $\chi_n:\mathbb{T}\to\mathbb{T}$ is given by $\chi_n(z):=z^n$ for $z\in\mathbb{T}$ and $n\in\mathbb{Z}$. 

\begin{proposition}
{\it Let $L$ be a LCA group and $\mathbb{W}(L)$ be the abstract Weyl-Heisenberg group associated to $L$. Then
\[
\Gamma(\mathbb{W}(L),\mathbb{T})=\mathbb{Z}.
\]
}\end{proposition}
\begin{proof}
Since $\mathbb{T}$ is central in $\mathbb{W}(L)$, for each $n\in\mathbb{Z}$, the character $\chi_n$ of $\mathbb{T}$ is $\mathbb{W}(L)$-invariant. Therefore, we get 
$\Gamma(\mathbb{W}(L),\mathbb{T})=\mathbb{Z}$.
\end{proof}
Suppose $n\in\mathbb{Z}$ and $\xi:=\chi_n$. Then, using (\ref{Jxi.N.Main}), for $f\in\mathcal{C}_c(\mathbb{W}(L))$ 
and $(x,\omega,z)\in\mathbb{W}(L)$, we get 
\[
T_\xi(f)(x,\omega,z)=\frac{z^n}{2\pi}\int_0^{2\pi}f(x,\omega,e^{\ii\alpha})e^{-\ii n\alpha}\dd\alpha.
\]
Therefore, each $\phi\in \mathcal{C}_\xi(\mathbb{W}(L),\mathbb{T})$ satisfies 
\begin{equation}\label{WL.main.prop.T}
\phi(x,\omega,z)=z^n\phi(x,\omega,1),
\end{equation}
for all $(x,\omega,z)\in\mathbb{W}(L)$.
Also, for each $1\le p<\infty$ and $\phi\in \mathcal{C}_\xi(\mathbb{W}(L),\mathbb{T})$, we have 
\[
\|\phi\|_{(p)}^p=\int_L\int_{\widehat{L}}|\phi(x,\omega,1)|^p\dd\lambda_{L}(x)\dd\lambda_{\widehat{L}}(\omega)=\|\phi\|_{L^p(\mathbb{W}(L))}^p.
\]
Let $\psi,\phi\in \mathcal{C}_\xi(\mathbb{W}(L),\mathbb{T})$ and $(x,\omega,z)\in\mathbb{W}(L)$. Then 
\begin{align*}
\psi\natural\phi(x,\omega,z)
&=\int_L\int_{\widehat{L}}\psi(x',\omega',1)\phi(x-x',\omega\overline{\omega'},\overline{\omega(x')}\omega'(x')z)\dd\lambda_{L}(x')\dd\lambda_{\widehat{L}}(\omega').
\end{align*}
Applying (\ref{WL.main.prop.T}), we get the following explicit characterization 
\begin{equation}\label{cc.main.semi.awhg.n=T}
\psi\natural\phi(x,\omega,z)=z^n\int_L\int_{\widehat{L}}\overline{\omega(x')}^n\omega'(x')^n\psi(x',\omega',1)\phi(x-x',\omega\overline{\omega'},1)\dd\lambda_{L}(x')\dd\lambda_{\widehat{L}}(\omega'),
\end{equation}
and the covariant involution is given by 
\begin{equation}\label{ci.main.semi.awhg.n=T}
\psi^\natural(x,\omega,z)=\omega(x)^{-n}z^n\overline{\psi(-x,\overline{\omega},1)}.
\end{equation}

\begin{example}
Let $d\ge 1$ and $L:=\mathbb{R}^d$. We then have $\widehat{\mathbb{R}^d}=\{\widehat{\bom}:\bom\in\mathbb{R}^d\}$, where the character $\widehat{\bom}:\mathbb{R}^d\to\mathbb{T}$ is given by $\widehat{\bom}(\mathbf{x}):=e^{\ii\langle\mathbf{x},\bom\rangle}$, for all $\mathbf{x}\in\mathbb{R}^d$. 
In this case, we have 
\[
\Gamma(\mathbb{R}^d\ltimes(\mathbb{R}^d\times\mathbb{T}),\mathbb{T})=\mathbb{Z}.
\]
Let $n\in\mathbb{Z}$, $\psi,\phi\in\mathcal{C}_{\chi_n}(\mathbb{W}(\mathbb{R}^d),\mathbb{T})$, and $(\mathbf{x},\bom,z)\in\mathbb{W}(\mathbb{R}^d)$. Then using (\ref{cc.main.semi.awhg.n=T}), we get 
\begin{equation*}
\psi\natural\phi(\mathbf{x},\bom,z)=z^n\int_{\mathbb{R}^d}\int_{\mathbb{R}^d}e^{\ii n\langle\bom'-\bom,\mathbf{x}'\rangle}\psi(\mathbf{x}',\bom',1)\phi(\mathbf{x}-\mathbf{x}',\bom-\bom',1)\dd\mathbf{x'}\dd\bom',
\end{equation*} 
and using (\ref{ci.main.semi.awhg.n=T}) the covariant involution is  
\[
\psi^\natural(\mathbf{x},\bom,z)=e^{-\ii n\langle\mathbf{x},\bom\rangle}z^n\overline{\psi(-\mathbf{x},-\bom,1)}.
\]
\end{example}
\begin{example}
Let $d\ge 1$ and $L:=\mathbb{Z}^d$. We then have 
$\widehat{\mathbb{Z}^d}=\{\widehat{\mathbf{z}}:\mathbf{z}=(z_1,...,z_d)^T\in\mathbb{T}^d\}$, where the character $\widehat{\mathbf{z}}:\mathbb{Z}^d\to\mathbb{T}$ is given by 
$\widehat{\mathbf{z}}(\mathbf{k}):=\prod_{j=1}^dz_j^{k_j}$, for $\mathbf{k}:=(k_1,...,k_d)^T\in\mathbb{Z}^d$ and $\mathbf{z}:=(z_1,...,z_d)^T\in\mathbb{T}^d$. In this case, we have 
\[
\Gamma(\mathbb{Z}^d\ltimes\mathbb{T}^{d+1},\mathbb{T})=\mathbb{Z}.
\]
Let $n\in\mathbb{Z}$, $\psi,\phi\in\mathcal{C}_{\chi_n}(\mathbb{Z}^d\ltimes\mathbb{T}^{d+1},\mathbb{T})$, and $(\mathbf{k},\mathbf{z},z)\in\mathbb{W}(\mathbb{Z}^d)$. Then, using (\ref{cc.main.semi.awhg.n=T}), we obtain
\[
\psi\natural\phi(\mathbf{k},\mathbf{z},z)=z^n\sum_{\mathbf{k}'\in\mathbb{Z}^d}\int_{\mathbb{T}^d}\widehat{\mathbf{z}}(\mathbf{k}')^{-n}\widehat{\mathbf{z}'}(\mathbf{k}')^n\psi(\mathbf{k}',\mathbf{z}',1)\phi(\mathbf{k}-\mathbf{k}',\mathbf{z}\overline{\mathbf{z}'},1)\dd\mathbf{z}',
\]
and using (\ref{ci.main.semi.awhg.n=T}), the covariant involution is given by 
\[
\psi^\natural(\mathbf{k},\mathbf{z},z)=\overline{\widehat{\mathbf{z}}(\mathbf{k})}^nz^n\overline{\psi(-\mathbf{k},\overline{\mathbf{z}},1)}.
\] 
\end{example}
\begin{example}
Let $M>0$ be a positive integer and $L:=\mathbb{Z}_M$ be the finite cyclic 
group of integers modulo $M$. We then have $\widehat{L}=\{\widehat{\ell}:\ell\in\mathbb{Z}_M\}$, where the character $\widehat{\ell}:\mathbb{Z}_M\to\mathbb{T}$ is given by 
$\widehat{\ell}(m):=e^{2\pi\ii m\ell/M}$, for every $m,\ell\in\mathbb{Z}_M$. In this case, we have 
\[
\Gamma(\mathbb{Z}_M\ltimes(\mathbb{Z}_M\times\mathbb{T}),\mathbb{T})=\mathbb{Z}.
\]
Let $n\in\mathbb{Z}$, $\psi,\phi\in\mathcal{C}_{\chi_n}(\mathbb{Z}_M\ltimes(\mathbb{Z}_M\times\mathbb{T}),\mathbb{T})$, and $(m,\ell,z)\in\mathbb{W}(\mathbb{Z}_M)$. Then, using (\ref{cc.main.semi.awhg.n=T}), we have 
\[
\psi\natural\phi(m,\ell,z)=z^n\sum_{m'=0}^{M-1}\sum_{\ell'=0}^{M-1}e^{2\pi\ii n(\ell'-\ell)m'/M}\phi(m',\ell',1)\varphi(m-m',\ell-\ell',1),
\]
and using (\ref{ci.main.semi.awhg.n=T}), the covariant involution is given by 
\[
\psi^\natural(m,\ell,z)=e^{-2\pi\ii nm\ell/M}z^n\overline{\psi(M-m,M-\ell,1)}.
\]
\end{example}

\subsubsection{\bf The Case $N=\widehat{L}\times\mathbb{T}$}
Let $N:=\{(e_L,\omega,z):\omega\in\widehat{L},z\in\mathbb{T}\}$.
Then, $\chi(N)=\{\chi_{y,n}:y\in L, n\in\mathbb{Z}\}$,  
where the character $\chi_{y,n}:N\to\mathbb{T}$ is given by 
$\chi_{y,n}(\omega,z):=\omega(y)z^n$, for all $(\omega,z)\in K=\widehat{L}\times\mathbb{T}$. 
\begin{proposition}\label{Gamma.WL.whL.T.L.infinite}
{\it Let $L$ be a locally compact Abelian group and $\mathbb{W}(L)$ be the abstract Weyl-Heisenberg group associated to $L$.  
\begin{enumerate}
\item If $L$ is an infinite group then $
\Gamma(\mathbb{W}(L),\widehat{L}\times\mathbb{T})=L$.
\item If $L$ is a finite group of order $M$ then  
$\Gamma(\mathbb{W}(L),\widehat{L}\times\mathbb{T})=L\times M\mathbb{Z}$.
\end{enumerate}
}\end{proposition}
\begin{proof}
Let $x,y\in L$ and $n\in\mathbb{Z}$. Then, for each $(\omega,z)\in\widehat{L}\times\mathbb{T}$, we have
\begin{align*}
\chi_{y,n}\circ\theta_x(\omega,z)
&=\chi_{y,n}(\omega,\omega(x)z)=\omega(y)\omega(x)^nz^n
\\&=\omega(y)\omega(x^n)z^n=\omega(yx^n)z^n=\chi_{yx^n,n}(\omega,z),
\end{align*}
which implies that $\chi_{y,n}\circ\theta_x=\chi_{yx^n,n}$. 
Invoking Proposition \ref{Ginv.semi.K}, we deduce that the character $\chi_{y,n}$ is $\mathbb{W}(L)$-invariant if and only if $y=yx^n$ for all $x\in L$ or equivalently $x^n=e_L$ for all $x\in L$. 

(1) Suppose $L$ is an infinite group. Then, $x^n=e_L$ for all $x\in L$, if and only if  
$n=0$. Hence, we conclude that the character $\chi_{y,n}$ is $\mathbb{W}(L)$-invariant if and only if $n=0$. Therefore, 
\[
\Gamma(\mathbb{W}(L),\widehat{L}\times\mathbb{T})=\{\chi_{y,0}:y\in L\}=L.
\]
(2) Let $L$ be a finite group of order $M$. Then, $x^n=e_L$ for all $x\in L$, if and only if $n\in M\mathbb{Z}$. Therefore, we deduce that $
\Gamma(\mathbb{W}(L),\widehat{L}\times\mathbb{T})=\{\chi_{y,n}:y\in L,\ n\in M\mathbb{Z}\}=L\times M\mathbb{Z}$.
\end{proof}

Let $L$ be an infinite LCA group and $\xi:=\chi_{y,0}$ with $y\in L$. Then, using (\ref{Jxi.K}), for $f\in\mathcal{C}_c(\mathbb{W}(L))$ 
and $(x,\omega,z)\in\mathbb{W}(L)$, we get 
\[
T_\xi(f)(x,\omega,z)=\frac{\omega(y)}{2\pi}\int_0^{2\pi}\int_{\widehat{L}}f(x,\zeta,e^{\ii\alpha})\overline{\zeta(y)}\dd\lambda_{\widehat{L}}(\zeta)\dd\alpha.
\]
Therefore, each $\phi\in \mathcal{C}_\xi(\mathbb{W}(L),\widehat{L}\times\mathbb{T})$ satisfies 
\begin{equation}\label{WL.main.prop.whL.T}
\phi(x,\omega,z)=\omega(y)\phi(x,\mathbf{1},1),
\end{equation}
for all $(x,\omega,z)\in\mathbb{W}(L)$.
Also, for each $1\le p<\infty$ and $\phi\in \mathcal{C}_\xi(\mathbb{W}(L),\widehat{L}\times\mathbb{T})$, we have 
\[
\|\phi\|_{(p)}^p:=\int_L|\phi(x,\mathbf{1},1)|^p\dd\lambda_{L}(x).
\]
Let $\psi,\phi\in \mathcal{C}_\xi(\mathbb{W}(L),\widehat{L}\times\mathbb{T})$ and $(x,\omega,z)\in\mathbb{W}(L)$. Then   
\begin{align*}
\psi\natural\phi(x,\omega,z)=
\int_L\phi(x',\mathbf{1},1)\varphi(x-x',\omega,\overline{\omega(x')}z)\dd\lambda_{L}(x'),
\end{align*}
Applying (\ref{WL.main.prop.whL.T}), we get the following explicit characterization 
\begin{equation}\label{cc.main.semi.awhg.n=whL.T}
\psi\natural\phi(x,\omega,z)=\omega(y)\int_L\psi(x',\mathbf{1},1)\phi(x-x',\mathbf{1},1)\dd\lambda_{L}(x'),
\end{equation}
and the covariant involution is given by 
\begin{equation}\label{ci.main.semi.awhg.n=whL.T}
\psi^\natural(x,\omega,z)=\omega(y)\overline{\psi(-x,\mathbf{1},1)}.
\end{equation}
\begin{example}
Let $d\ge 1$ and $L:=\mathbb{R}^d$. We then have $\widehat{\mathbb{R}^d}=\{\widehat{\bom}:\bom\in\mathbb{R}^d\}$, where the character $\widehat{\bom}:\mathbb{R}^d\to\mathbb{T}$ is given by $\widehat{\bom}(\mathbf{x}):=e^{\ii\langle\mathbf{x},\bom\rangle}$, for $\mathbf{x}\in\mathbb{R}^d$. In this case, using Proposition \ref{Gamma.WL.whL.T.L.infinite}(1), we have 
\[
\Gamma(\mathbb{R}^d\ltimes(\mathbb{R}^d\times\mathbb{T}),\mathbb{R}^d\times\mathbb{T})=\mathbb{R}^d.
\]
Let $\mathbf{y}\in\mathbb{R}^d$ and $\xi=\chi_{\mathbf{y},0}$. Suppose $\psi,\phi\in\mathcal{C}_{\xi}(\mathbb{W}(\mathbb{R}^d),\mathbb{R}^d\times\mathbb{T})$ and $( \mathbf{x},\bom,z)\in\mathbb{W}(\mathbb{R}^d)$. Using  (\ref{cc.main.semi.awhg.n=whL.T}), the covariant convolution is   
\begin{equation*}
\psi\natural\phi(\mathbf{x},\bom,z)=e^{\ii\langle\bom,\mathbf{y}\rangle}\int_{\mathbb{R}^d}\psi(\mathbf{x}',\mathbf{1},1)\phi(\mathbf{x}-\mathbf{x}',\mathbf{1},1)\dd\mathbf{x'},
\end{equation*} 
and using (\ref{ci.main.semi.awhg.n=whL.T}), the covariant involution is 
\[
\psi^\natural(\mathbf{x},\bom,z)=e^{\ii\langle\bom,\mathbf{y}\rangle}\overline{\psi(-\mathbf{x},\mathbf{1},1)}.
\]
\end{example}
\begin{example}
Let $d\ge 1$ and $L:=\mathbb{Z}^d$. We then have 
$\widehat{\mathbb{Z}^d}=\{\widehat{\mathbf{z}}:\mathbf{z}=(z_1,...,z_d)^T\in\mathbb{T}^d\}$, where the character $\widehat{\mathbf{z}}:\mathbb{Z}^d\to\mathbb{T}$ is given by 
$\widehat{\mathbf{z}}(\mathbf{k}):=\prod_{j=1}^dz_j^{k_j}$, for all $\mathbf{k}:=(k_1,...,k_d)^T\in\mathbb{Z}^d$ and $\mathbf{z}:=(z_1,...,z_d)^T\in\mathbb{T}^d$. In this case, using Proposition \ref{Gamma.WL.whL.T.L.infinite}(1), we have
\[
\Gamma(\mathbb{Z}^d\ltimes\mathbb{T}^{d+1},\mathbb{T}^{d+1})=\mathbb{Z}^d.
\]
Let $\mathbf{n}\in\mathbb{Z}^d$ and $\xi:=\chi_{\mathbf{n}}$. Suppose $\psi,\phi\in\mathcal{C}_{\xi}(\mathbb{W}(\mathbb{Z}^d),\mathbb{T}^{d+1})$ and $(\mathbf{k},\mathbf{z},z)\in\mathbb{W}(\mathbb{Z}^d)$. Then, using (\ref{cc.main.semi.awhg.n=whL.T}), the covariant convolution is   
\[
\psi\natural\phi(\mathbf{k},\mathbf{z},z)=\widehat{\mathbf{z}}(\mathbf{n})\sum_{\mathbf{k}'\in\mathbb{Z}^d}\psi(\mathbf{k}',\mathbf{1},1)\phi(\mathbf{k}-\mathbf{k}',\mathbf{1},1),
\]
and using (\ref{ci.main.semi.awhg.n=whL.T}), the covariant involution is  
\[
\psi^\natural(\mathbf{k},\mathbf{z},z)=\widehat{\mathbf{z}}(\mathbf{n})\overline{\psi(-\mathbf{k},\mathbf{1},1)}.
\]
\end{example}
Let $L$ be a finite Abelian group of order $M$ and $\xi:=\chi_{y,n}$ with $y\in L$ and $n\in M\mathbb{Z}$. Then, $\widehat{L}$ is a finite Abelian group as well. Thus, using (\ref{Jxi.K}), for $f\in\mathcal{C}_c(\mathbb{W}(L))$ 
and $(x,\omega,z)\in\mathbb{W}(L)$, we get 
\[
T_\xi(f)(x,\omega,z)=\frac{\omega(y)z^n}{2\pi}\int_0^{2\pi}\sum_{\zeta\in\widehat{L}}f(x,\zeta,e^{\ii\alpha})\overline{\zeta(y)}e^{-\ii n\alpha}\dd\alpha.
\]
Therefore, each $\phi\in \mathcal{C}_\xi(\mathbb{W}(L),\widehat{L}\times\mathbb{T})$ satisfies 
\begin{equation}\label{WL.main.prop.whL.T.L.finite}
\phi(x,\omega,z)=\omega(y)z^n\phi(x,\mathbf{1},1),
\end{equation}
for all $(x,\omega,z)\in\mathbb{W}(L)$.
Also, for each $1\le p<\infty$ and $\phi\in \mathcal{C}_\xi(\mathbb{W}(L),\widehat{L}\times\mathbb{T})$, we have 
\[
\|\phi\|_{(p)}^p:=\sum_{x\in L}|\phi(x,\mathbf{1},1)|^p.
\]
Let $\psi,\phi\in \mathcal{C}_\xi(\mathbb{W}(L),\widehat{L}\times\mathbb{T})$ and $(x,\omega,z)\in\mathbb{W}(L)$. Then    
\begin{align*}
\psi\natural\phi(x,\omega,z)=
\sum_{x'\in L}\phi(x',\mathbf{1},1)\varphi(x-x',\omega,\overline{\omega(x')}z).
\end{align*}
Applying (\ref{WL.main.prop.whL.T.L.finite}), we get the following explicit characterization 
for covariant convolution 
\begin{equation}\label{cc.main.semi.awhg.n=whL.T.L.finite}
\psi\natural\phi(x,\omega,z)=\omega(y)z^n\sum_{x'\in L}\psi(x',\mathbf{1},1)\phi(x-x',\mathbf{1},1),
\end{equation}
and the covariant involution is given by 
\begin{equation}\label{ci.main.semi.awhg.n=whL.T.L.finite}
\psi^\natural(x,\omega,z)=\omega(y)z^n\overline{\psi(-x,\mathbf{1},1)}.
\end{equation}
\begin{example}
Let $M>0$ be a positive integer and $L:=\mathbb{Z}_M$ be the finite cyclic 
group of integers modulo $M$. We then have $\widehat{L}=\{\widehat{\ell}:\ell\in\mathbb{Z}_M\}$, where the character $\widehat{\ell}:\mathbb{Z}_M\to\mathbb{T}$ is given by 
$\widehat{\ell}(m):=e^{2\pi\ii m\ell/M}$, for all $m,\ell\in\mathbb{Z}_M$. In this case, using Proposition \ref{Gamma.WL.whL.T.L.infinite}(2), we get 
\[
\Gamma(\mathbb{Z}_M\ltimes(\mathbb{Z}_M\times\mathbb{T}),\mathbb{Z}_M\times\mathbb{T})=\mathbb{Z}_M\times M\mathbb{Z}.
\] 
Let $n\in\mathbb{Z}_M$ and $k\in M\mathbb{Z}$. Suppose $\xi:=\chi_{n,k}$, $\psi,\phi\in \mathcal{C}_\xi(\mathbb{W}(\mathbb{Z}_M),\mathbb{Z}_M\times\mathbb{T})$, and $(m,\ell,z)\in\mathbb{W}(\mathbb{Z}_M)$. Then, using (\ref{cc.main.semi.awhg.n=whL.T.L.finite}), the covariant convolution is given by  
\[
\psi\natural\phi(m,\ell,z)=z^ke^{2\pi\ii\ell n/M}\sum_{m'=0}^{M-1}\psi(m',\mathbf{1},1)\phi(m-m',\mathbf{1},1),
\]
and using (\ref{ci.main.semi.awhg.n=whL.T.L.finite}) the covariant involution is given by 
\[
\psi^\natural(m,\ell,z)=z^ke^{2\pi\ii\ell n/M}\overline{\psi(M-m,\mathbf{1},1)}.
\]
\end{example}

\subsection{The Heisenberg Groups}
Let $d\ge 1$ and locally compact groups $H,K$ be given by $H:=\mathbb{R}^d$
and $K:=\mathbb{R}^d\times\mathbb{R}$. For each $\mathbf{x}\in H$, define $\theta_\mathbf{x}:K\to K$ by 
$\theta_\mathbf{x}(\mathbf{y},s):=(\mathbf{y},s+\langle\mathbf{x},\mathbf{y}\rangle)$,
for $(\mathbf{y},s)\in K$. The Heisenberg group $\mathbb{H}^d$ is the semi-direct product group $H\ltimes_\theta K$, that is the underlying manifold $\mathbb{R}^{2d+1}$ with the group law given by 
\[
(\mathbf{x},\mathbf{y},s)\ltimes_\theta (\mathbf{x}',\mathbf{y}',s')=(\mathbf{x}+\mathbf{x}',\mathbf{y}+\mathbf{y}',s+s'+\langle\mathbf{x},\mathbf{y}'\rangle),
\] 
\[
(\mathbf{x},\mathbf{y},s)^{-1}=(-\mathbf{x},-\mathbf{y},-s+\langle\mathbf{x},\mathbf{y}\rangle),
\] 
for every $(\mathbf{x},\mathbf{y},s),(\mathbf{x}',\mathbf{y}',s')\in\mathbb{H}^d$.
Then $\chi(K)=\{\mathbf{e}_{\mathbf{z},\nu}:\nu\in\mathbb{R},\mathbf{z}\in\mathbb{R}^d\}$, where 
\[
\mathbf{e}_{\mathbf{z},\nu}(\mathbf{y},s)
:=e^{\ii\langle\mathbf{z},\mathbf{y}\rangle}e^{\ii \nu s},
\]
for $s,\nu\in\mathbb{R}$ and $\mathbf{z},\mathbf{y}\in\mathbb{R}^d$.

\subsubsection{\bf The Case $N=\mathbb{R}$}
Let $N:=\{(\mathbf{0},\mathbf{0},s):s\in\mathbb{R}\}$. Then $N$ is a central subgroup of $\mathbb{H}^d$ and also 
$\theta_\mathbf{x}(\mathbf{0},s)=(\mathbf{0},s)$, for $\mathbf{x}\in H$ and $s\in\mathbb{R}$. Thus, any character $\xi\in\chi(\mathbb{R})$ is $\mathbb{H}^d$-invariant. Therefore, we get 
$\Gamma(\mathbb{H}^d,\mathbb{R})=\chi(\mathbb{R})$.  
Also, we have $\chi(\mathbb{R})=\{\mathbf{e}_\nu:\nu\in\mathbb{R}\}$, where $\mathbf{e}_\nu:=\mathbf{e}_{\mathbf{0},\nu}$. We then have 
$\Gamma(\mathbb{H}^d,\mathbb{R})=\mathbb{R}$. 
Let $\nu\in\mathbb{R}$ and $\xi:=\mathbf{e}_\nu$. Suppose $f\in\mathcal{C}_c(\mathbb{H}^d)$ and $(\mathbf{x},\mathbf{y},t)\in\mathbb{H}^d$. Then 
\[
T_\nu(f)(\mathbf{x},\mathbf{y},t)=\int_{-\infty}^\infty f(\mathbf{x},\mathbf{y},t+s)e^{-\ii\nu s}\dd s,
\]
Hence, 
\[
T_\nu(f)(\mathbf{x},\mathbf{y},t)=e^{\ii\nu t}\int_{-\infty}^\infty f(\mathbf{x},\mathbf{y},s)e^{-\ii \nu s}ds=e^{\ii \nu t}T_\nu(f)(\mathbf{x},\mathbf{y},0).
\]
So, each $\psi\in\mathcal{C}_\nu(\mathbb{H}^d,\mathbb{R})$ satisfies 
\begin{equation}\label{Hd.main.prop.R}
\psi(\mathbf{x},\mathbf{y},t)=e^{\ii\nu t}\psi(\mathbf{x},\mathbf{y},0).
\end{equation}
Let $\psi,\phi\in\mathcal{C}_\nu(\mathbb{H}^d,\mathbb{R})$. Then, using (\ref{cc.main.semi}), the covariant convolution is given by 
\begin{align*}
\psi\natural\phi(\mathbf{x},\mathbf{y},t)
=\int_{\mathbb{R}^d}\int_{\mathbb{R}^d}\psi(\mathbf{x}',\mathbf{y}',0)\phi(\mathbf{x}-\mathbf{x}',\mathbf{y}-\mathbf{y}',t-\langle\mathbf{x}',\mathbf{y}-\mathbf{y}'\rangle)\dd\mathbf{x}'\dd\mathbf{y}'.
\end{align*}
Applying (\ref{Hd.main.prop.R}), we get the following explicit characterization
\begin{equation*}
\psi\natural\phi(\mathbf{x},\mathbf{y},t)=e^{\ii\nu t}\int_{\mathbb{R}^d}\int_{\mathbb{R}^d}\psi(\mathbf{x}',\mathbf{y}',0)\phi(\mathbf{x}-\mathbf{x}',\mathbf{y}-\mathbf{y}',0)e^{-\ii\nu\langle\mathbf{x}',\mathbf{y}-\mathbf{y}'\rangle}\dd\mathbf{x}'\dd\mathbf{y}',
\end{equation*}
and the covariant involution is given by 
\[
\psi^\natural(\mathbf{x},\mathbf{y},t)=e^{\ii\nu s}e^{-\ii\nu\langle\mathbf{x},\mathbf{y}\rangle}\overline{\psi(-\mathbf{x},-\mathbf{y},0)}.
\]
\subsubsection{\bf The Case $N=\mathbb{R}^{d+1}$}
Let $N:=\{(\mathbf{y},s):s\in\mathbb{R},\mathbf{y}\in\mathbb{R}^d\}$.
Then, $N$ is a closed normal subgroup of $\mathbb{H}^d$. 
The following result characterizes $\mathbb{H}^d$-invariant characters of $\mathbb{R}^{d+1}$. 

\begin{proposition}
{\it Let $d\ge 1$ be given. Then  
\[
\Gamma(\mathbb{H}^d,\mathbb{R}^{d+1})=\mathbb{R}^d.
\]
}\end{proposition}
\begin{proof}
Let $\nu\in\mathbb{R}$ and $\mathbf{z}\in\mathbb{R}^d$ be given. Then
\begin{align*}
\mathbf{e}_{\mathbf{z},\nu}\circ\theta_{\mathbf{x}}(\mathbf{y},s)
&=\mathbf{e}_{\mathbf{z},\nu}(\theta_{\mathbf{x}}(\mathbf{y},s))
=\mathbf{e}_{\mathbf{z},\nu}(\mathbf{y},s+\langle\mathbf{x},\mathbf{y}\rangle)
=e^{\ii\langle\mathbf{z},\mathbf{y}\rangle}e^{\ii \nu(s+\langle\mathbf{x},\mathbf{y}\rangle)}
\\&=e^{\ii\langle\mathbf{z},\mathbf{y}\rangle}e^{\ii \nu s}e^{\ii\nu\langle\mathbf{x},\mathbf{y}\rangle}
=e^{\ii\langle\mathbf{z},\mathbf{y}\rangle}e^{\ii\langle\nu\mathbf{x},\mathbf{y}\rangle}e^{\ii \nu s}=e^{\ii\langle\mathbf{z}+\nu\mathbf{x},\mathbf{y}\rangle}e^{\ii \nu s}=\mathbf{e}_{\mathbf{z}+\nu\mathbf{x},\nu}(\mathbf{y},s),
\end{align*}
for every $\mathbf{x},\mathbf{y}\in\mathbb{R}^d$ and $s\in\mathbb{R}$.
Therefore, we obtain $\mathbf{e}_{\mathbf{z},\nu}\circ\theta_{\mathbf{x}}=\mathbf{e}_{\mathbf{z}+\nu\mathbf{x},\nu}$, for every $\mathbf{x}\in\mathbb{R}^d$. 
Using Proposition \ref{Ginv.semi.K}, we conclude that $\mathbf{e}_{\mathbf{z},\nu}\in\Gamma(\mathbb{H}^d,\mathbb{R}^{d+1})$ if and only if $\mathbf{e}_{\mathbf{z},\nu}=\mathbf{e}_{\mathbf{z}+\nu\mathbf{x},\nu}$, for every $\mathbf{x}\in\mathbb{R}^d$. So $\mathbf{e}_{\mathbf{z},\nu}\in\Gamma(\mathbb{H}^d,\mathbb{R}^{d+1})$ if and only if $\mathbf{z}+\nu\mathbf{x}=\mathbf{z}$ for every $\mathbf{x}\in\mathbb{R}^d$. This implies that $\mathbf{e}_{\mathbf{z},\nu}\in\Gamma(\mathbb{H}^d,\mathbb{R}^{d+1})$ if and only if $\nu=0$.
\end{proof} 
Let $\mathbf{z}\in\mathbb{R}^d$ and $\xi:=\mathbf{e}_\mathbf{z}$, where $\mathbf{e}_\mathbf{z}:=\mathbf{e}_{\mathbf{z},0}$.
Suppose $f\in\mathcal{C}_c(\mathbb{H}^d)$ and $(\mathbf{x},\mathbf{y},t)\in\mathbb{H}^d$. Then  
\[
J_\mathbf{z}(f)(\mathbf{x},\mathbf{y},t)=e^{\ii\langle\mathbf{z},\mathbf{y}\rangle}\int_{-\infty}^\infty\int_{\mathbb{R}^d} f(\mathbf{x},\mathbf{y}',s')e^{-\ii\langle\mathbf{z},\mathbf{y}'\rangle}\dd s' \dd\mathbf{y}',
\]
Hence, we get 
\[
J_\mathbf{z}(f)(\mathbf{x},\mathbf{y},t)=e^{\ii\langle\mathbf{z},\mathbf{y}\rangle}J_\mathbf{z}(f)(\mathbf{x},\mathbf{0},0).
\] 
Thus, each $\psi\in\mathcal{C}_\mathbf{z}(\mathbb{H}^d,\mathbb{R}^{d+1})$ satisfies 
\begin{equation}\label{Hd.main.prop.RRd}
\psi(\mathbf{x},\mathbf{y},t)=e^{\ii\langle\mathbf{z},\mathbf{y}\rangle}\psi(\mathbf{x},\mathbf{0},0).
\end{equation}
Let $\psi,\phi\in\mathcal{C}_\mathbf{z}(\mathbb{H}^d,\mathbb{R}^{d+1})$ and $(\mathbf{x},\mathbf{y},t)\in\mathbb{H}^d$. Then, using (\ref{cc.main.semi.canoK}), the covariant convolution is  
\begin{align*}
\psi\natural\phi(\mathbf{x},\mathbf{y},t)
=\int_{\mathbb{R}^d}\psi(\mathbf{x}',\mathbf{0},0)\phi(\mathbf{x}-\mathbf{x}',\mathbf{y},t-\langle\mathbf{x}',\mathbf{y}\rangle)\dd\mathbf{x}'.
\end{align*}
Applying (\ref{Hd.main.prop.RRd}), we get the following explicit characterization for covariant convolution  
\begin{equation*}
\psi\natural\phi(\mathbf{x},\mathbf{y},t)=e^{\ii\langle\mathbf{z},\mathbf{y}\rangle}\int_{\mathbb{R}^d}\psi(\mathbf{x}',\mathbf{0},0)\phi(\mathbf{x}-\mathbf{x}',\mathbf{0},0)\dd\mathbf{x}'.
\end{equation*}
and the covariant involution is given by 
\[
\psi^\natural(\mathbf{x},\mathbf{y},t)=e^{\ii\langle\mathbf{z},\mathbf{y}\rangle}\overline{\psi(-\mathbf{x},\mathbf{0},0)}.
\]

\subsection{The Subgroups of Jacobi Groups}

Let $d\ge 1$ and $\mathrm{Sp}(d)$ be the real symplectic group. 
Throughout this section, we briefly present basics of classical harmonic analysis on the real symplectic group $\Sp(d)$, for more details we referee the readers to \cite{deg0, kisil.book} and the comprehensive list of references therein. 

A matrix $\bfS\in M_{2d\times 2d}(\mathbb{R})$ is called symplectic if and only if $\bfS^T\mathbf{J}\bfS=\bfS\mathbf{J}\bfS^T=\mathbf{J}$, with 
$$\mathbf{J}=\left(\begin{array}{cc}
\mathbf{0} & {\mathbf{I}_{d}} \\
{-\mathbf{I}_{d}} & \mathbf{0} \\
\end{array}\right),$$ 
where 
$\mathbf{I}_{d}$ is $d\times d$ identity matrix..
The group consists of all symplectic matrices is called the (real) symplectic group which is denoted by $\Sp(d)$. 
It is a simple non-compact finite-dimensional real Lie group. 
In block-matrix notation, the symplectic group $\Sp(d)$ 
consists of all real $2d\times 2d$ matrices in block form 
\[
\bfS=\left(\begin{array}{cc}
\mathbf{A} & \mathbf{B} \\
\mathbf{C} & \mathbf{D} \\
\end{array}\right),
\hspace{1cm}\mathbf{A},\mathbf{B},\mathbf{C},\mathbf{D}\in M_{d\times d}(\mathbb{R}),
\]
such that $\mathbf{A}^T\mathbf{C}=\mathbf{C}^T\mathbf{A}$, $\mathbf{B}^T\mathbf{D}=\mathbf{D}^T\mathbf{B}$, and $\mathbf{A}^T\mathbf{D}-\mathbf{C}^T\mathbf{B}=\mathbf{I}_{d}$.
 
The real symplectic group $\Sp(d)$ satisfies the following decomposition, namely Iwasawa (Gram-Schmidt) decomposition, $\Sp(d)=\mathcal{K}\mathcal{A}\mathcal{N}$ where \cite{deg0, AGHF.JPhyA}
\begin{equation}
\mathcal{K}_d:=\left\{\Omega(\mathbf{A}+\ii\mathbf{B})=\left(\begin{array}{cc}
\mathbf{A} & \mathbf{-B} \\
\mathbf{B} & \mathbf{A} \\
\end{array}\right):\mathbf{A}+\ii\mathbf{B}\in\mathrm{U}(d,\mathbb{C})\right\},
\end{equation} 
\begin{equation}
\mathcal{A}_d:=\left\{{\rm diag}(h_1,...,h_d,h_1^{-1},...,h_d^{-1}):h_1,...,h_d>0\right\},
\end{equation} 
and 
\begin{equation}
\mathcal{N}_d:=\left\{
\left(\begin{array}{cc}
\mathbf{A} & \mathbf{B} \\
\mathbf{0} & \mathbf{A}^{-T} \\
\end{array}\right):\mathbf{A}\ {\rm is\ unit\ upper\ triangular},\ \mathbf{AB}^T=\mathbf{BA}^T\right\}.
\end{equation}

Regarding elements of $\Sp(d)$ as linear transformations over the vector space $\mathbb{R}^{2d}={\mathbb{R}^d}\times\widehat{{\mathbb{R}^d}}$, then the symplectic group 
$\Sp(d)$ is precisely the group of all linear automorphisms of 
${\mathbb{R}^d}\times\widehat{{\mathbb{R}^d}}$ which 
preserve the canonical (symplectic) form. 
Also, it is easy to check that, if $\lambda_{{\mathbb{R}^d}\times\widehat{{\mathbb{R}^d}}}$ is the Lebesgue measure on ${\mathbb{R}^d}\times\widehat{{\mathbb{R}^d}}$, then 
$\mathrm{d}\lambda_{{\mathbb{R}^d}\times\widehat{{\mathbb{R}^d}}}(\bfS\bfmu)=\mathrm{d}\lambda_{{\mathbb{R}^d}\times\widehat{{\mathbb{R}^d}}}(\bfmu)
$, for $\bfS\in\Sp(d)$.

Let $H$ be a closed subgroup of the real symplectic group $\mathrm{Sp}(d)$. Assume that $K$ is the underlying manifold $\mathbb{R}^{2d}\times\mathbb{R}$ equipped with the following group law
\[
(\bfmu,t)\ltimes(\bfmu',t'):=(\bfmu+\bfmu',t+t'+\frac{1}{2}\mathcal{J}[\bfmu,\bfmu']),
\]
\[
(\bfmu,t)^{-1}:=(-\bfmu,-t),
\]
where the symplectic form $\mathcal{J}:\mathbb{R}^{2d}\times\mathbb{R}^{2d}\to\mathbb{R}$ is $\mathcal{J}[\bfmu,\bfmu']:=\bfmu^T\mathbf{J}\bfmu'$,
for all $\bfmu,\bfmu'\in\mathbb{R}^{2d}$ with 
\[
\mathbf{J}:=\left(\begin{matrix}
\mathbf{0} & \mathbf{I}_d \\ 
-\mathbf{I}_d & \mathbf{0} \\ 
\end{matrix}\right).
\]
Suppose $\mathbb{G}_d(H):=H\ltimes_\theta K$ where the homomorphism $\theta:H\to Aut(K)$ is given by $\bfS\mapsto\theta_\bfS$, with $\theta_\bfS(\bfmu,t):=(\bfS\bfmu,t)$ for $(\bfmu,t)\in\mathbb{R}^{2d}\times\mathbb{R}$.
Then the group $\mathbb{G}_d(H)$ has the underlying manifold $H\times \mathbb{R}^{2d+1}=H\times\mathbb{R}^{2d}\times\mathbb{R}$ with the following group laws
\[
(\bfS,\bfmu,t)\ltimes(\bfS',\bfmu',t'):=(\bfS\bfS',\bfmu+\bfS\bfmu',t+t'+\frac{1}{2}\mathcal{J}[\bfmu,\bfS\bfmu']),
\]
\[
(\bfS,\bfmu,t)^{-1}:=(\bfS^{-1},-\bfS^{-1}\bfmu,-t),
\]

It is shown that if $\lambda_H$ is a left Haar measure of $H$ then 
$\dd\lambda_{\mathbb{G}_d(H)}(\bfS,\bfmu,t):=\dd\lambda_H(\bfS)\dd\bfmu\dd t,$ 
is a left Haar mesure for $\mathbb{G}_d(H)$ and if $\Delta_H$ is the modular function of $H$ then $\Delta_{\mathbb{G}_d(H)}(\bfS,\bfmu,t):=\Delta_H(\bfS)$ is the modular function of $\mathbb{G}_d(H)$, see \cite{AGHF.JPhyA}. 
\subsubsection{\bf The Case $N:=\mathbb{R}$}
Let $N:=\{(\mathbf{I},\mathbf{0},s):s\in\mathbb{R}\}$. Then, $N$ is a central subgroup of $\mathbb{G}_d(H)$. So $N$ is normal in $\mathbb{G}_d(H)$. 
\begin{proposition}
{\it Let $d\ge 1$ be given. Then  
\[
\Gamma(\mathbb{G}_d(H),\mathbb{R})=\mathbb{R}.
\]
}\end{proposition}
\begin{proof}
Since $\mathbb{R}$ is a central subgroup of $\mathbb{G}_d(H)$, we deduce that each character $\xi\in\chi(\mathbb{R})$ is 
$\mathbb{G}_d(H)$-invariant. Hence, we get 
$\Gamma(\mathbb{G}_d(H),\mathbb{R})=\chi(\mathbb{R})$.
Since $\chi(\mathbb{R})=\mathbb{R}$, each character $\xi\in\chi(\mathbb{R})$ has the form $\chi_\nu$ for a unique $\nu\in\mathbb{R}$, where the character $\chi_\nu:\mathbb{R}\to\mathbb{T}$ is given by $\chi_\nu(t):=e^{\ii \nu t}$. 
\end{proof}
Let $\nu\in\mathbb{R}$, $f\in\mathcal{C}_c(\mathbb{H}^d)$, and $(\bfS,\bfmu,t)\in\mathbb{G}_d(H)$. Then
\[
J_\nu(f)(\bfS,\bfmu,t)=\int_{-\infty}^\infty f(\bfS,\bfmu,t+s)e^{-\ii\nu s}\dd s.
\]
Hence, we get 
\[
J_\nu(f)(\bfS,\bfmu,t)=e^{\ii\nu t}\int_{-\infty}^\infty f(\bfS,\bfmu,s)e^{-\ii \nu s}\dd s=e^{\ii \nu t}J_\nu(f)(\bfS,\bfmu,0).
\]
Thus, each $\psi\in\mathcal{C}_\nu(\mathbb{G}_d(H),\mathbb{R})$ satisfies 
\begin{equation}\label{GHd.main.prop.R}
\psi(\bfS,\bfmu,t)=e^{\ii\nu t}\psi(\bfS,\bfmu,0).
\end{equation}
In this case, for $1\le p<\infty$, the norm is given by 
\[
\|\psi\|_{(p)}^p=\int_H\int_{\mathbb{R}^{2d}}|\psi(\bfS,\bfmu,0)|^p\dd\lambda_H(\bfS)\dd\bfmu.
\]
Let $\psi,\phi\in\mathcal{C}_\nu(\mathbb{G}_d(H),\mathbb{R})$ and $(\bfS,\bfmu,t)\in\mathbb{G}_d(H)$. Then, using (\ref{cc.main.semi}), we have 
\begin{align*}
\psi\natural\phi(\bfS,\bfmu,t)
&=\int_{H}\int_{\mathbb{R}^{2d}}\psi(\bfS',\bfmu',0)\phi(\bfS'^{-1}\bfS,\bfS'^{-1}\bfmu-\bfS'^{-1}\bfmu',t-\frac{1}{2}\mathcal{J}[\bfmu',\bfmu])\dd\lambda_{H}(\bfS')\dd\bfmu'.
\end{align*}
Applying (\ref{GHd.main.prop.R}), we get the following explicit characterization for covariant convolution 
\begin{equation}\label{cc.GHd.N=R}
\psi\natural\phi(\bfS,\bfmu,t)=e^{\ii\nu t}\int_{H}\int_{\mathbb{R}^{2d}}e^{-\ii\nu\frac{1}{2}\mathcal{J}[\bfmu',\bfmu]}\psi(\bfS',\bfmu',0)\phi(\bfS'^{-1}\bfS,\bfS'^{-1}\bfmu-\bfS'^{-1}\bfmu',0)\dd\lambda_{H}(\bfS')\dd\bfmu',
\end{equation}
and the covariant involution is 
\begin{equation}
\psi^\natural(\bfS,\bfmu,t):=\Delta_H(\bfS^{-1})e^{\ii\nu t}\overline{\psi(\bfS^{-1},-\bfS^{-1}\bfmu,0)}.
\end{equation}

\begin{example}
Let $H:=\{\mathbf{I}_{2d}\}$ be the identity subgroup. In this case, $\mathbb{G}_d(\{\mathbf{I}_{2d}\})\cong\mathbb{H}^d$ via the topological group isomorphism $(\bfmu,t)\to(\bfmu,t+\frac{1}{2}\langle\mathbf{x},\mathbf{y}\rangle)$, for every $\bfmu=(\mathbf{x},\mathbf{y})\in\mathbb{R}^{2d}$ and $t\in\mathbb{R}$.  Then covariant convolution/involution 
on $\mathcal{C}_\nu(\mathbb{G}_d(\{\mathbf{I}_{2d}\}),\mathbb{R})$ can be identified via the covariant convolution/involution on 
$\mathcal{C}_\nu(\mathbb{H}^d,\mathbb{R})$, for every $\nu\in\mathbb{R}$. 
\end{example}

\begin{example}
Let $H=\Sp(d)$ be the real symplectic group. Then, $\mathbb{G}_d(\Sp(d))$ is the Jacobi group $\mathbb{G}^d$. Using (\ref{cc.GHd.N=R}), we get the following explicit formula for covariant convolutions 
\begin{equation*}
\psi\natural\phi(\bfS,\bfmu,t)=e^{\ii\nu t}\int_{\Sp(d)}\int_{\mathbb{R}^{2d}}e^{-\ii\nu\frac{1}{2}\mathcal{J}[\bfmu',\bfmu]}\psi(\bfS',\bfmu',0)\phi(\bfS'^{-1}\bfS,\bfS'^{-1}\bfmu-\bfS'^{-1}\bfmu',0)\dd\lambda_{\Sp(d)}(\bfS')\dd\bfmu',
\end{equation*}
and the covariant involution is 
\[
\psi^\natural(\bfS,\bfmu,t):=\Delta_H(\bfS^{-1})e^{\ii\nu t}\overline{\psi(\bfS^{-1},-\bfS^{-1}\bfmu,0)}.
\]
\end{example}

\begin{example}
Let $H$ be the subgroup of $\Sp(d)$ given by 
\[
H:=\left\{\widetilde{\mathbf{H}}:=\left(\begin{matrix}
\mathbf{H} & \mathbf{0} \\ 
\mathbf{0} & \mathbf{H}^{-T}
\end{matrix}\right):\mathbf{H}\in\mathrm{GL}(d,\mathbb{R})\right\}.
\]
Then the group $\mathbb{G}_d(H)$ is isomorphic to 
$\mathrm{GL}(d,\mathbb{R})\ltimes\mathbb{R}^{2d+1}$ with the group law 
\[
(\mathbf{H},\mathbf{p},\mathbf{q},t)\ltimes(\mathbf{H}',\mathbf{p}',\mathbf{q}',t')
=(\mathbf{HH}',\mathbf{p}+\mathbf{H}\mathbf{p}',\mathbf{q}+\mathbf{H}^{-T}\mathbf{q}',t+t'+\frac{1}{2}(\mathbf{p}^T\mathbf{H}^{-T}\mathbf{q}'-\mathbf{q}^{T}\mathbf{H}\mathbf{p}')).
\]
Let $\nu\in\mathbb{R}$, $\psi,\phi\in\mathcal{C}_\nu(\mathbb{G}_d(H),\mathbb{R})$, and $(\widetilde{\mathbf{H}},\bfmu,t)\in\mathbb{G}_d(H)$. Then, using (\ref{cc.GHd.N=R}), $\psi\natural\phi(\widetilde{\mathbf{H}},\bfmu,t)$ is 
\[
e^{\ii\nu t}\int_{\mathrm{GL}(d,\mathbb{R})}\int_{\mathbb{R}^{2d}}e^{-\ii\nu\frac{1}{2}\mathcal{J}[\bfmu',\bfmu]}\psi(\widetilde{\mathbf{H}'},\bfmu',0)\phi(\widetilde{\mathbf{H}'^{-1}\mathbf{H}},\widetilde{\mathbf{H}'^{-1}}\bfmu-\widetilde{\mathbf{H}'^{-1}}\bfmu',0)|\det(\mathbf{H}')|^{-d}\dd\mathbf{H}'\dd\bfmu',
\]
where $\dd\mathbf{H}'$ is the 
Lebesgue measure on $\mathrm{GL}(d,\mathbb{R})$. The covariant involution is 
\[
\psi^\natural(\mathbf{H},\mathbf{p},\mathbf{q},t):=\Delta_H(\mathbf{H}^{-1})e^{\ii\nu t}\overline{\psi(\mathbf{H}^{-1},-\mathbf{H}^{-1}\mathbf{p},-\mathbf{H}^T\mathbf{q},0)}.
\]
\end{example}

\begin{example}
Let $H$ be the subgroup of $\Sp(d)$ given by 
\[
H:=\left\{\mathbf{S}(\mathbf{h}):=
\left(\begin{array}{cc}
\mathrm{diag}(\mathbf{h}) & \mathbf{0} \\
\mathbf{0} & \mathrm{diag}(\mathbf{h})^{-1} \\
\end{array}\right):\mathbf{h}\in\mathbb{R}^d_+\right\}.
\]
In this case, the group $\mathbb{G}_d(H)$ is isomorphic to $\mathbb{R}^d_+\ltimes\mathbb{R}^{2d+1}$ with the group law   
\[
(\mathbf{h},\mathbf{p},\mathbf{q},t)\ltimes(\mathbf{h}',\mathbf{p}',\mathbf{q}',t')
=(\mathbf{h}\odot\mathbf{h}',\mathbf{p}+\mathbf{h}\odot\mathbf{p}',\mathbf{q}+\mathbf{h}^{-1}\odot\mathbf{q}',t+t'+\frac{\langle\mathbf{p},\mathbf{h}^{-1}\odot\mathbf{q}'\rangle-\langle\mathbf{q},\mathbf{h}\odot\mathbf{p}'\rangle}{2}),
\]
where $\odot$ is the Hadamard product. Let $\nu\in\mathbb{R}$, $\psi,\phi\in \mathcal{C}_\nu(\mathbb{G}_d(H),\mathbb{R})$, and $(\mathbf{h},\bfmu,t)\in\mathbb{G}_d(H)$. Then, using (\ref{cc.GHd.N=R}), the  covariant convolution is 
\[
\psi\natural\phi(\mathbf{h},\bfmu,t)=e^{\ii\nu t}\int_{\mathbb{R}_+^d}\int_{\mathbb{R}^{2d}}e^{-\ii\nu\frac{1}{2}\mathcal{J}[\bfmu',\bfmu]}\psi(\mathbf{h}',\bfmu',0)\phi(\mathbf{h}'^{-1}\odot\mathbf{h},\mathbf{S}(\mathbf{h}'^{-1})\bfmu-\mathbf{S}(\mathbf{h}'^{-1})\bfmu',0)\dd\mathbf{h}'\dd\bfmu',
\]
and the covariant involution is 
\[
\psi^\natural(\mathbf{h},\bfmu,t):=e^{\ii\nu t}\overline{\psi(\mathbf{h}^{-1},
-\mathbf{S}(\mathbf{h})^{-1}\bfmu,0)}.
\]
\end{example}

{\bf Acknowledgment.}
This project has received funding from the European Union’s Horizon 2020 research and innovation programme under the Marie Sklodowska-Curie grant agreement No. 794305.  The author gratefully acknowledge the supporting agencies.  The findings and opinions expressed here are only those of the author, and not of the funding agency.\\
The author would like to express his deepest gratitude to Vladimir V. Kisil for suggesting the problem that motivated the results in this article, stimulating discussions and pointing out various references. 

\bibliography{CovInv}
\bibliographystyle{amsplain}
\end{document}